\documentclass[11pt]{article}
\usepackage{tikz}
\usetikzlibrary{intersections, calc, arrows.meta}
\usepackage{latexsym, amsmath, amssymb, euscript, amsthm, a4, epsfig}
\usepackage[all]{xy}
\usepackage{enumerate}
\usepackage{graphicx}
\usepackage{color}
\usepackage{comment}
\usepackage{mathrsfs}
\usepackage[dvipdfmx]{hyperref}

\theoremstyle{definition}
\newtheorem{theorem}{Theorem}[section]
\newtheorem{corollary}[theorem]{Corollary}
\newtheorem{lemma}[theorem]{Lemma}
\newtheorem{proposition}[theorem]{Proposition}
\newtheorem{definition}[theorem]{Definition}
\newtheorem{notation}[theorem]{Notation}

\newtheorem{example}[theorem]{Example}
\newtheorem{remark}[theorem]{Remark}

\newtheorem{conjecture}[theorem]{Conjecture}

\newcommand{\eqnref}[1]{(\ref {#1})}

\newcommand{\Rbb}{\mathbb{R}}

\newcommand{\Vbb}{\mathbb{V}}
\newcommand{\Zbb}{\mathbb{Z}}

\newcommand{\Ccal}{\mathcal{C}}

\newcommand{\Hcal}{\mathcal{H}}
\newcommand{\Ical}{\mathcal{I}}

\newcommand{\Pcal}{\mathcal{P}}

\def\Bk{{\bf k}}

\def\BK{{\bf K}}

\def\BN{{\bf N}}

\newcommand{\Ga}{\alpha}

\newcommand{\Gvf}{\varphi}
\newcommand{\Gg}{\gamma}

\newcommand{\Gn}{\eta}

\newcommand{\Gt}{\theta}

\newcommand{\Gr}{\rho}

\newcommand{\Gy}{\psi}

\newcommand{\GD}{\Delta}
\newcommand{\GF}{\Phi}
\newcommand{\GG}{\Gamma}
\newcommand{\GL}{\Lambda}

\newcommand{\GO}{\Omega}

\newcommand{\GY}{\Psi}

\newcommand{\beqn}{\begin{equation}}
\newcommand{\eeqn}{\end{equation}}
\newcommand{\beqy}{\begin{eqnarray*}}
\newcommand{\eeqy}{\end{eqnarray*}}
\newcommand{\bey}{\begin{eqnarray}}
\newcommand{\eey}{\end{eqnarray}}
\numberwithin{equation}{section}

\def\({\Bigl(}
\def\){\Bigr)}
\def\8{\biggl(}
\def\9{\biggr)}
\def\2{\biggl\{}
\def\3{\biggr\}}
\newcommand{\tcv}{\ \widetilde{*}\ }
\newcommand{\tcp}{\ \widetilde{\circ}\ }

\def\tX{\widetilde{X}}
\def\tp{\widetilde{p}}
\def\tq{\widetilde{q}}

\newcommand{\cm}[1]{\ar@{}@<-1ex>[#1]|{\circlearrowright}}
\newcommand{\pt}{\mbox{pt}}
\newcommand{\Set}{\mathcal{S}et}
\newcommand{\Sh}{\mbox{Sh}}

\newcommand{\Vect}{\mbox{Vect}}

\newcommand{\supp}{\mbox{supp}}

\renewcommand{\SS}{\mbox{SS}}

\newcommand{\inte}{\mbox{Int}}

\newcommand{\Hom}{\mbox{Hom}}
\newcommand{\HOM}{\Hcal om}
\newcommand{\RHom}{\mbox{RHom}}
\newcommand{\RHOM}{\mbox{R}\Hcal om}

\newcommand{\Ker}{\mbox{Ker\! }}
\newcommand{\Coker}{\mbox{Coker\! }}
\renewcommand{\Im}{\mbox{Im\! }}
\newcommand{\ow}{\mbox{otherwise}}
\newcommand{\Ob}{\mbox{Ob}}
\newcommand{\id}{\mbox{id}}

\newcommand{\ctext}[1]
{\raise0.2ex\hbox{\textcircled{\scriptsize{#1}}}}
\newcommand{\indlim}[1]
{\mbox{“}\varinjlim_{#1}\hspace{-0.5em}\mbox{”}}
\newcommand{\prolim}[1]
{\mbox{“}\varprojlim_{#1}\hspace{-0.5em}\mbox{”}}

\def\ol{\overline}
\def\ul{\underline}
\def\bs{\backslash}

\newcommand{\bpape}[5]{%
	\bibitem [#1]{#1} #2. {\it #3}. #4. #5. 		% お好みに合わせて変えてください。
}
\newcommand{\bpaper}[6]{%
	\bibitem [#1]{#1} #2. {\it #3}. #4. #5. #6. 			% お好みに合わせて変えてください。
}
\newcommand{\bbook}[6]{%
	\bibitem [#1]{#1} #2. {\it #3}. #4. #5. #6.			% お好みに合わせて変えてください。
}
\newcommand{\barxiv}[5]{%
	\bibitem [#1]{#1} #2. {\it #3}. Available at \href{#4}{#5}.  			% お好みに合わせて変えてください。
}

\allowdisplaybreaks

\title{An isometry theorem induced by the Radon transform between the convolution and interleaving distances}

\author{Michiaki Takiwaki}

\begin{document}

%\newpage

\maketitle

\begin{abstract}
	One-parameter persistence modules are applied to various subjects as tools in data analysis. On the other hand, since the theoretical study of multi-parameter persistence modules is not enough and in progress, they have few applications.  
	The sheaf theory is expected to elucidate detailed properties of persistence modules and give features of multi-parameter ones for applications. 
	However, the categories of sheaves on two or more dimensional Euclidean spaces have more complicated structures than those on $\Rbb$. %convolution? 
	The Radon transform for sheaves is a useful dimension reduction technique and induces a categorical equivalence between the localized bounded derived categories of sheaves on $\Rbb^n$ and those on $S^{n-1}\times\Rbb$. We show 
	We develop the convolution and the interleaving distances on these localized categories by improving original distances on the derived categories of sheaves on $\Rbb^n$ and those on $S^{n-1}\times\Rbb$. The convolution bifunctor defines these distances. 
	We show that the Radon transform changes multi-directional movements given by the convolution bifunctor to one-directional movements and induce an isometry theorem between these distances.  

\end{abstract}

\clearpage

\tableofcontents
\clearpage

\section{Introduction}
\label{Intro}\phantom{a} \vspace{-1em}
	\subsection{Related work and main result}
	\label{rwmr}\phantom{a}
		One-parameter persistence modules are applied to various subjects as tools in data analysis. On the other hand, since the theoretical study of multi-parameter persistence modules is not enough and in progress, they have few applications and require further features to describe real-world data. In order to obtain tractable features, we explore mathematical objects which correspond to multi-parameter persistence modules. 
		
		%sheaf theory and persistent homology
		There are close relations between the sheaf theory and the theory of persistence modules. For a partially ordered set $\Pcal$, Curry \cite{Cur14} found an equivalence of categories from the category of persistence modules on $\Pcal$ to the category $\Sh(\Pcal_{\mathfrak a})$ of sheaves on $\Pcal$ with the Alexandrov topology (see \cite[Definition 4.2.2]{Cur14}) for $\Pcal$. 
		For a vector space $\Vbb$ endowed with the Euclidean topology, Kashiwara and Schapira \cite{KS18} showed that a cone $\Gg$ in $\Vbb$ satisfying suitable conditions gives a partially ordered structure to $\Vbb$ and defined a functor from $\Sh(\Vbb_{\mathfrak a})$ to the category $\Sh(\Vbb_\Gg)$ of $\Gg$-sheaves which is a full subcategory of the category of sheaves on $\Vbb$. Berkouk and Petit \cite{BP21} proved that the functor used in \cite{KS18} induces an equivalence of categories from a quotient category of $\Sh(\Vbb_{\mathfrak a})$ by the category of ephemeral modules (see \cite[Definition 3.4]{BP21}) to $\Sh(\Vbb_\Gg)$. These equivalences of categories enable us to replace some problems of persistence modules with those of sheaves. In particular, multi-parameter persistence modules correspond to sheaves on $\Rbb^n\ (n\geq2)$. 
				
		%\Rbb^nと\Rbb^nの比較
		However, the bounded derived category $D^b(\Rbb^n)$ of sheaves on $\Rbb^n$ has more complicated structures than that on $\Rbb$. Every constructible sheaf on $\Rbb$ is decomposable (see \cite[Proposition 2.17.]{KS18}, \cite[Corollary 7.3.]{Gui16}), but constructible sheaves on $\Rbb^n$ may not be. Kashiwara and Schapira \cite{KS18} defined a bifunctor $(-)*(-):D^b(\Vbb)\times D^b(\Vbb)\to D^b(\Vbb)$ which is called the convolution bifunctor. They constructed a sheaf $K_a$ in $D^b(\Vbb)$ for $a\in\Rbb$ and defined the convolution distance $d_C$(see Definition \ref{a-isomorphism}) on $D^b(\Vbb)$ by applying the functor $K_a*(-)$. Precisely, the convolution distance is an extended pseudo-distance. We need some conditions about sheaves for proving that $d_C(F, G)=0$ if and only if $F\simeq G$ in $D^b(\Vbb)$. 
		Since the constructibility condition for $D^b(\Rbb)$ gives the decomposability theorem, $d_C$ on the derived category of constructible sheaves on $\Rbb$ satisfies the metric axioms (see \cite[Theorem 6.3.]{BG22}). On the other hand, a decomposability theorem for the category of constructible sheaves on $\Rbb^n$ is not yet found. For example, we need a stronger condition to prove that $d_C$ on $D^b(\Rbb^n)$ satisfies the metric axioms (see \cite{PSW24}, \cite{GV24}). 
	 
		 %projected barcode
		 
		 There are some techniques for dealing with $D^b(\Rbb^n)$. The projected barcode \cite{BP22} is a functor which reduces a sheaf to a tractable object. The idea of the projected barcode is derived from the fibered barcode (see \cite{Lan18}, \cite{LW15}) for persistence modules. The both barcodes are dimension reduction techniques defined by functions between $\Rbb^n$ and $\Rbb$. However we do not know how the projected barcode keeps properties of sheaves. In fact, there is a question about the convolution distance and the liner integral sheaf metric which is an extended pseudo-distance defined by the projected barcode (see \cite[Question 6.18.]{BP22}). 

		 %Radon transform
		 
		 Another useful technique is the Radon transform (see \cite[Definition 3.1.]{Gao17}), which is similar to the projected barcode. It is a functor from the derived category of sheaves on $\Rbb^n$ to that on $S^{n-1}\times \Rbb$. The Radon transform induces an equivalence of categories given by quantized contact transforms between the localized derived categories. A notable feature of the Radon transform is that original sheaves on $\Rbb^n$ can be restored up to constant sheaves from sheaves on $S^{n-1}\times\Rbb$. 
		 
		In this paper, we prove an isometry theorem by the Radon transform. 
		We first construct localized extended pseudo-distance $d_{LC}$ (resp.\,$d_{LI}$) on the localized categories of $D^b(\Rbb^n)$ (resp.\,$D^b(S^{n-1}\times\Rbb)$) from the usual convolution distance $d_C$ (resp.\,interleaving distance $d_I$). Here $d_I$ is an extended pseudo-distance on $D^b(S^{n-1}\times\Rbb)$, which is similar to $d_C$(see Definition \ref{a-interleaved}). The main theorem of this paper is an isometry theorem between $d_{LC}$ and $d_{LI}$. 

		\begin{theorem}
		\label{main theorem}
			The Radon transform $\GF_{\Bk_A}$ induces an isometric functor between the localized convolution and interleaving distances:
			\begin{equation}
				\GF_{\Bk_A}:(D^b(\Rbb^n, \dot{T}^*\Rbb^n), d_{LC})\to (D^b(S^{n-1}\times\Rbb, T^{*, +}(S^{n-1}\times\Rbb)), d_{LI}). 
			\end{equation}
			Namely, for any $F, G\in D^b(\Rbb^n, \dot{T}^*\Rbb^n)$, $d_{LC}(F, G)=d_{LI}(\GF_{\Bk_A}(F), \GF_{\Bk_A}(G))$. 
		\end{theorem}
		
		This theorem is expected to solve problems mentioned above about the convolution distance.
		The convolution bifunctor transforms complexes of sheaves on $\Vbb$. In particular, The functor $K_{a}*(-)$ which is necessary for the definition of $d_C$ gives multi-directional movements to sheaves on $\Rbb^n$. If the behavior of $K_{a}*(-)$ causes a difficulty associated with the convolution distance, the above theorem is a powerful tool to clarify such a problem, because the Radon transform changes multi-directional movements to one-directional movements. 
		
		In order to observe such behavior, let us consider an example given by the subsets
		\begin{align*}
			S&=\{(x_1, x_2)\in\Rbb^2; -1\leq x_1\leq 1, -1\leq x_2\leq 1\}, &\\
			T&=\{(x_1, 1)\in\Rbb^2; -1\leq x_1\leq 1\}\sqcup\{(x_1, -1)\in\Rbb^2; -1\leq x_1\leq 1\}&
		\end{align*}
		of $\Rbb^2$. Let $\Bk_Z$ denote a complex of sheaves  in $D^b(\Rbb^2)$ defined by the constant sheaf on $\Rbb^2$ and the subset $Z=S\bs T$. Here $\Bk_Z$ satisfies that the restricted sheaf $\Bk_Z|_Z$ is isomorphic to the constant sheaf on $Z$ and $\Bk_Z|_{\Rbb^2\bs Z}\simeq 0$. For each $a\geq0$, we have the complex of sheaf $K_a*\Bk_Z$. If $0\leq a\leq 1$, there is a subset $Z^\prime_a$ of $\Rbb^2$ such that $K_a*\Bk_Z\simeq \Bk_{Z^\prime_a}$ If $a>1$, there are two subsets $Z^\prime_a$ and $Z^{\prime\prime}_a$ of $\Rbb^2$ such that $H^0(K_a*\Bk_Z)\simeq \Bk_{Z^\prime_a}$ and $H^1(K_a*\BK_Z)\simeq\Bk_{Z^{\prime\prime}_a}$. 
		
		Applying the Radon transform, $\GF_{\Bk_A}(K_a*\Bk_Z)$ is simpler than $K_a*\Bk_Z$. For any $a\geq0$, there is a continuous function $\Gvf:S^1\to\Rbb$ such that $\GF_{\Bk_A}(K_a*\Bk_Z)\simeq\Bk_{\{(y, t)\in S^1\times\Rbb; \Gvf(y)\leq t+a\}}[-1]$. 
		
		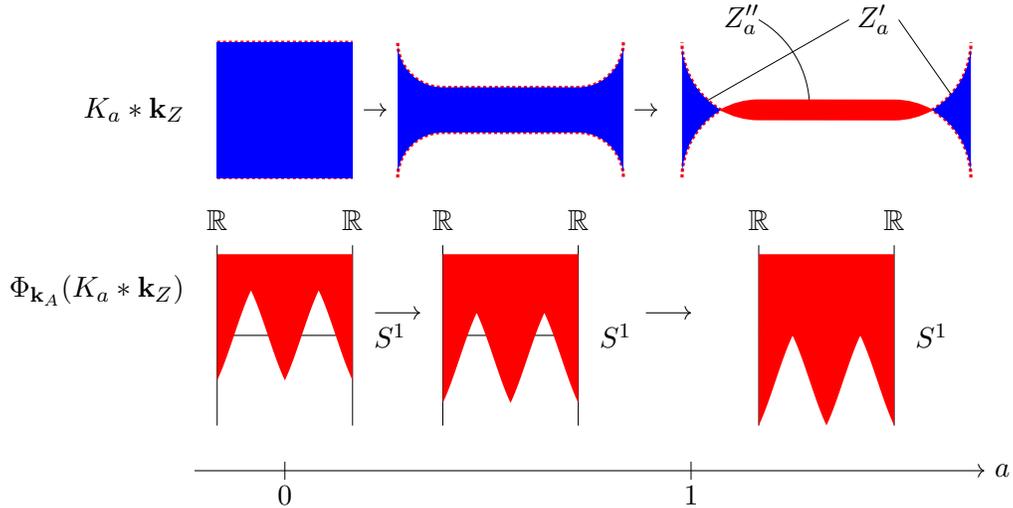
\begin{figure}[htbp]
		\label{movement_example}
			\centering
			\begin{tikzpicture}[scale=0.3]
				\coordinate[label=below:$0$](0)at(0, -16.2);
				\coordinate[label=below:$1$](1)at(18, -16.2);
				\coordinate[label=right:$a$](a)at(31, -16);
				\coordinate[label=left:$K_a*\Bk_Z$](KF)at(-4, 0);
				\coordinate[label=left:$\GF_{\Bk_A}(K_a*\Bk_Z)$](RKF)at(-4, -8);
				\coordinate[label=right:$Z^{\prime\prime}_a$](Z’a)at(19,4);
				\coordinate[label=right:$Z^\prime_a$](Z’’a)at(24.9,4);
				\draw[very thin] (0, -15.6) -- (0, -16.4);
				\draw[very thin] (18, -15.6) -- (18, -16.4);
				\draw[->] (3.5, 0) -- (4.5, 0);
				\draw[->] (15.5, 0) -- (16.5, 0);
				\draw[red, very thick, densely  dotted] (-3, 3) -- (3, 3);
				\draw[red, very thick, densely  dotted] (-3, -3) -- (3, -3);
				\draw[red, very thick, densely  dotted] (5, -3)arc(180:90:2)--(13, -1)arc(90:0:2);
				\draw[red, very thick, densely  dotted] (15, 3)arc(0:-90:2)--(7, 1)arc(270:180:2);
				\draw[red, very thick, densely  dotted] (17.6, -3)arc(180:120:3.46)arc(240:180:3.46);
				\draw[red, very thick, densely  dotted] (30.4, -3)arc(0:60:3.46)arc(-60:0:3.46);
				\draw[->] (4, -9) -- (6, -9);
				\draw[->] (16, -9) -- (18, -9);
				\draw[->] (-4, -16) -- (31, -16);
				\draw[very thin] (21, 4) arc (60:0:4.5); 
				\draw[very thin] (18, 0)--(25, 4);
				\draw[very thin] (30, 0)--(27.2, 4);
				\fill[blue, opacity=0.4] (-3, 3)--(-3, -3)--(3, -3)--(3, 3)--cycle;
				\fill[blue, opacity=0.4] (7, 1)arc(270:180:2)--(5, -3)arc(180:90:2)--(13, -1)arc(90:0:2)--(15, 3)arc(0:-90:2)--cycle;
				\fill[blue, opacity=0.4] (17.6, -3)arc(180:120:3.46)arc(240:180:3.46)--cycle;
				\fill[blue, opacity=0.4] (30.4, -3)arc(0:60:3.46)arc(-60:0:3.46)--cycle;
				\fill[red, opacity=0.4] (21, -0.46)arc(270:240:3.46)arc(120:90:3.46)--(27, 0.46)arc(90:60:3.46)arc(-60:-90:3.46)--cycle;
				%\fill[blue, opacity=0.4] (-3, 2)arc(135:225:2.8284)--(-2, -3)arc(225:315:2.8284)--(3, -2)arc(-45:45:2.8284)--(2, 3)arc(45:135:2.8284)--cycle;
				%\fill[gray, opacity=0.4] (14, -7) circle (3mm);
				%\fill[blue, opacity=0.4] (8, -7) circle (14.142mm);
				%\fill[blue, opacity=0.4] (0, -7) circle (28.284mm);
				\foreach\a in {-3, 7, 21}{
				\coordinate[label=below:$\Rbb$](Rbb1)at(\a, -4);
				\coordinate[label=below:$\Rbb$](Rbb1)at(\a+6, -4);
				\coordinate[label=right:$S^1$](S1)at(\a+6.5, -10);
				\draw[very thin] (\a, -14) -- (\a, -6);
				\draw[very thin] (\a +6, -14) -- (\a +6, -6);
				\draw[thin] (\a, -10) -- (\a+6, -10);
				\fill[red, opacity=0.4]
					plot[domain=0:90,smooth]({\a+(\x)/60},{2.82*sin((\x)-45)-10-(\a>6)-(\a>20)})
					--plot[domain=90:180,smooth]({\a+(\x)/60},{2.82*sin((\x)+45)-10-(\a>6)-(\a>20)}) 
					--plot[domain=180:270,smooth]({\a+(\x)/60},{2.82*sin((\x)+135)-10-(\a>6)-(\a>20)}) 
					--plot[domain=270:360,smooth]({\a+(\x)/60},{2.82*sin((\x)+225)-10-(\a>6)-(\a>20)})
					--(\a+6, -6.4)--(\a, -6.4)--cycle; 
				}
			\end{tikzpicture}
			
			\caption{For each $a\geq0$, the upper shape is $Z^\prime_a$ or $Z^\prime_a \cup Z^{\prime\prime}_a$ and the lower shape is $\{(y, t)\in S^1\times\Rbb; \Gvf(y)\leq t+a\}$. Since $S^1$ is a quotient space of the closed interval $[0, 1]$ of $\Rbb$ by identifying its endpoints $\{0, 1\}$, we show $\{(y, t)\in S^1\times\Rbb; \Gvf(y)\leq t+a\}$ as a subset of $[0, 1]\times\Rbb$ on this figure. }
		\end{figure}
		
		Increasing $a\geq0$, $Z^\prime_a$ is separated into two connected components and $H^1(K_a*\Bk_{Z})$ appears. On the other hand, increasing $a\geq0$, $\{(y, t)\in S^1\times\Rbb; \Gvf(y)\leq t+a\}$ just moves in the direction of $\Rbb$. Then by using the Radon transform, the movements given by the functor $K_a*(-)$ can be replaced with the tractable movements. Therefore, we expect that the property of the convolution distance $d_C$ on $D^b(\Rbb^n)$ is described by that on $D^b(\Rbb)$ or $d_I$ on $D^b(S^{n-1}\times\Rbb)$. 
		
		We will show some conjectures about the convolution and interleaving distances and remind this example in Section.\ref{conjecture} to explain a relation between these conjectures. 
		
	\subsection{Structure of this paper}
	\label{Str}\phantom{a} \vspace{-1em}%1/12ここまで
		
		This paper is organized as follows. 
		In Section.\ref{Pre}, we show some basics of sheaf theory. We define Grothendieck\!’\!\!s six operations: $\RHOM, \otimes, Rf_*, f^{-1}, Rf_!, f^!$ from continuous maps $f:X\to Y$ satisfying suitable conditions. The composition bifunctor which is necessary for the definition of the Radon transform is composed of these functors.
		In Section.\ref{QCT}, we show the microlocal theory for sheaves and the quantized contact transform. The microsupport of sheaf is necessary for localizations of categories. The Radon transform gives a special  equivalence of categories induced by the composition of the quantized contact transform. 
		In Section.\ref{CD}, we define the convolution functor and compare it with the composition functor. Applying the convolution bifunctor, we define the convolution distance and the interleaving distance. These distances are also defined in the setting of localized categories of sheaves by using an inequality of microsupports. 
		In Section.\ref{IT}, we prove the main theorem and show some conjectures related with the distances of sheaves. In addition, we give an example of the Radon transform changing multi directional movements to one directional movements. 

	{\bf Acknowledgments}
	
		I would like to express my sincere gratitude to my supervisor Yasuaki Hiraoka for his useful advices about approaches to research. I also thank Kohei Yahiro, who is a postdoctoral researcher of the same lab as me, for some comments about basics of sheaf theory from the preliminary stage of my study. 
		
		I am also grateful to Yuichi Ike for many discussions and sharing with me some topics relevant to my research. I would also like to thank Tomohiro Asano and Tatsuki Kuwagaki for helpful comments from points of view of symplectic geometry. This research is supported by JST, CREST Grant Number JPMJCR24Q1, Japan. 
		
\section{Preliminaries}
\label{Pre}\phantom{a}
	%sheaf theory
	Throughout this paper, $\Bk$ is a field and $X$ is a finite dimensional real smooth manifold. 
	
	\subsection{Presheaf and sheaf}
	\label{presheaf and sheaf}\phantom{a}
	%この論文ではk-ベクトル空間に値をとる層のみ考える
		A presheaf $F$ on $X$ is a contravariant functor from the category of open subsets of $X$ to the category $\Vect_\Bk$ of $\Bk$-vector space.  For an open subset $U$ of $X$, an element $s\in F(U)$ is called a {\bf section} of $F$. For an open inclusion $V\subset U$, a morphism $\Gr^F_{V, U}:F(U)\to F(V)$ defined by the inclusion is called the {\bf restriction morphism}. A morphism of presheaves $\Gvf: F\to G$ is a natural transformation from $F$ to $G$. Namely, it is a family of linear maps $\{\Gvf_U:F(U)\to G(U)\}$ and satisfies the following commutativety
		\[
		\xymatrix{
			F(U)\ar[r]^{\Gvf_{U}}\ar[d]_{\Gr^F_{V, U}}\ar@{}[dr]|{\circlearrowright}&G(U)\ar[d]^{\Gr^G_{V, U}}\\
			F(V)\ar[r]_{\Gvf_V}&G(V)
		}
		\]
		for all pair of open subsets $V\subset U$. For a section $s\in F(U)$, we often write $s|_V$ instead of $\Gr^F_{V, U}(s)$. Let $\mbox{PSh}(X)$ denote the category of presheaves of $\Bk$-vector space on $X$. 
	
		\begin{definition}
		\label{sheaf}
			A presheaf $F$ on $X$ is called a {\bf sheaf} if it satisfies the {\bf gluing condition} $(*)$ below. 
			
			$(*)$: For any open set $U\subset X$, any open covering $U=\bigcup_{i\in I}U_i$ and any family $s_i\in F(U_i)$ satisfying $s_i|_{U_i\cap U_j}=s_j|_{U_i\cap U_j}$ for all pairs $(i, j)$, there uniquely exists $s\in F(U)$ such that $s|_{U_i}=s_i$ for all $i$. 
		\end{definition}
		
		By the gluing condition, $F(\emptyset)=0$ for any sheaf $F$. 
		We define a morphism of sheaves as a morphism of the underlying presheaves. Let $\Sh(X)$ denote the category of sheaves of $\Bk$-vector space on $X$. %$\Sh(X)$ $D(X)$ (resp. $D^b(X)$) by the derived category (bounded derived category) of $\Sh(X)$. 
		
		\begin{definition}
		\label{stalk}
			Let $F$ be a sheaf on $X$. One defines the {\bf stalk} of $F$ at $x\in X$ as follow:
			\begin{equation}
				F_x:=\varinjlim_{U} F(U), 
			\end{equation}
			where $U$ runs through the family of open neighborhoods of $x$. 
		\end{definition}
		
		For a open subset $U\subset X$, we denote by $s_x$ the image of a section $s\in F(U)$ by the canonical morphism $F(U)\to F_x$. 
		The stalk of $F$ at $x\in X$ gives information of $F$ around an open neighborhood of $x$. It is useful to determine whether a morphism of sheaves is isomorphism or not.  
		
		\begin{proposition}\cite[Proposition 2.2.2.]{KS90}
		\label{stalk_iso}
			Let $\Gvf:F\to G$ be a morphism of sheaves on $X$. Then $\Gvf$ is an isomorphism if and if only for any $x\in X$, the induced morphism $\Gvf_x:F_x\to G_x$ is an isomorphism. 
		\end{proposition}
		\begin{proof}
			Assume that $\Gvf$ is an isomorphism of sheaves. For any open subset $U\subset X$, $\Gvf_U:F(U)\to G(U)$ is an isomorphism of vector spaces. Therefore, for any $x\in X$, $\Gvf_x:F_x\to G_x$ is isomorphism. 
			
			Conversely assume $\Gvf_x$ is an isomorphism for all $x\in X$. Let $s$ be a section $F$ on an open subset $U$ of $X$ satisfying $\Gvf_U(s)=0$. Since $\Gvf_x(s_x)=(\Gvf_U(s))_x=0$, we have $s_x=0$ for all $x\in U$. By the gluing condition, we obtain $s=0$. Thus $\Gvf$ is an injective. Let $t$ be a section of $G$ on an open subset $U$ of $X$. By the assumption, there are an open covering $\{U_i\}$ of $U$ and $s_i\in F(U_i)$ such that $\Gvf_{U_i}(s_i)=t|_{U_i}$. Since 
			\[
			\Gvf_{U_i\cap U_j}(s_i|_{U_i\cap U_j})=\Gvf_{U_i}(s_i)|_{U_i\cap U_j}=t|_{U_i\cap U_j}=\Gvf_{U_j}(s_j)|_{U_i\cap U_j}=\Gvf_{U_i\cap U_j}(s_j|_{U_i\cap U_j}), 
			\]
			the injectivity of $\Gvf$ implies $s_i|_{U_i\cap U_j}=s_j|_{U_i\cap U_j}$ and hence there exists $s\in F(U)$ such that $s|_{U_i}=s_i$. We can check that $\Gvf_U(s)=t$. Thus $\Gvf$ is surjective.  
		\end{proof}
		
		\begin{proposition}\cite[Proposition 2.2.3.]{KS90}
		\label{sheafification}
			Given a presheaf $F$ on $X$, there exists a sheaf $F^+$ on $X$ and a morphism of presheaves $\Gt:F\to F^+$ such that for any sheaf $G$ and any morphism of presheaf $\Gvf:F\to G$, there uniquely exists a morphism $\widetilde\Gvf:F^+\to G$ which satisfies the following commutative diagram
			\[
			\xymatrix{
				F\ar[r]^{\Gvf}\ar[d]_{\Gt}&G\ar@{}@<-2ex>[dl]|{\circlearrowright}\\
				F^+. \ar@{.>}[ur]_{\widetilde\Gvf}&
			}
			\]
		\end{proposition}
		\begin{proof}
			For any open subset $U$ of $X$, we define the set of functions: 
			\begin{eqnarray*}
				F^+(U)&:=&\2s:U\to \prod_{x\in U}F_x; \text{ for any } x\in U, s(x)\in F_x \text{ and there is an open neighborhood } V \text{ of } x, \\
				&&V\subset U \text{ and } t\in F(V) \text{ such that } t_y=s(y) \text{ for any } y\in V\3. 
			\end{eqnarray*}
			Then $F^+$ satisfies the gluing condition. Therefore it is a sheaf on $X$. 
			
			We shall define the morphism $\Gt: F\to F^+$ as below: 
			\[
				\Gt_U:F(U)\to F^+(U),\quad s\mapsto\8s:U\to \prod_{x\in U}F_x, x\mapsto \{s_x\}\9, 
			\]
			where $U$ is an open subset of $X$. It satisfies that $\Gt_x:F_x\to (F^+)_x$ is an isomorphism for any $x\in X$. In particular if $G$ is a sheaf, $\Gt:G\to G^+$ is an isomorphism by Proposition \ref{stalk_iso}. Thus we can obtain the morphism $\widetilde\Gvf$ as the image of $\Gvf$ by the following homomorphism
			\[
				\Hom_{\mbox{PSh}(X)}(F, G)\to \Hom_{\Sh(X)}(F^+, G^+)\simeq\Hom_{\Sh(X)}(F^+, G). 
			\]
			It is easily checked that the above homomorphism has an inverse. Then it implies the uniqueness of $\widetilde\Gvf$. 
		\end{proof}
		We call $F^+$ the {\bf associated sheaf} to the presheaf $F$. 
		
		\begin{example}
		\label{constant sheaf}
			In the general setting, presheaves on $X$ are not sheaves. For example, if $X=\Rbb$ and $\Bk=\Zbb/2\Zbb$, we define a presheaf $F$ on $\Rbb$ which satisfies that for two open subsets $U$ and $V$ of $\Rbb$, 
			\begin{align*}
				U&\mapsto \Zbb/2\Zbb, \\
				V\subset U&\mapsto \id_{\Zbb/2\Zbb}:F(U)\to F(V). 
			\end{align*} 
			For $U_0=(-1, 0)$ and $U_1=(0, 1)$, we can not find $s\in F(U_0\cup U_1)$ which satisfies $s|_{U_0}=0\in F(U_0)=\Zbb/2\Zbb$ and $s|_{U_1}=1\in F(U_1)=\Zbb/2\Zbb$. 
		\end{example}
		\begin{example}
			\begin{enumerate}[(i)]
				\item We denote by $C^0_X$ a presheaf $U\mapsto \{f:U\to \Rbb; f {\rm \ is\ continuous\ on\ }U\}$ and by $\Gr^{C^0_X}_{V, U}$ a restriction morphism for the usual restriction of a function. Then $C^0_X$ is a sheaf on $X$. 
				\item We denote by $\Bk^\vee_X$ a presheaf $U\mapsto \{f:U\to \Bk; f {\rm \ is \ constant \ on\ }U\}$ and by $\Gr^{\Bk^\vee_X}_{V, U}$ a restriction morphism for the usual restriction of a function. Then $\Bk^\vee_X$ is not a sheaf on $X$. 
				\item We denote by $\Bk_X$ a presheaf $U\mapsto \{f:U\to \Bk; f {\rm \ is \ locally\ constant \ on\ }U\}$ and by $\Gr^{\Bk_X}_{V, U}$ a restriction morphism for the usual restriction of a function. Then $\Bk_X$ is a sheaf on $X$ since it is an assciated sheaf to $\Bk^\vee_X$. 
			\end{enumerate}
		\end{example}
		
		\begin{definition}
		\label{kernel and cokernel of sheaves}
			Let $\Gvf: F\to G$ be a morphism of sheaves on $X$. 
			\begin{enumerate}[(i)]
				\item One defines $\Ker\Gvf$ as a presheaf $U\mapsto \Ker\Gvf_U$. 
				\item One defines $\Coker\Gvf$ as an associated sheaf to the presheaf $U\mapsto\Coker\Gvf_U$. 
			\end{enumerate}
		\end{definition}
		
		Since the presheaf $U\mapsto \Ker\Gvf_U$ satisfies the gluing condition, it is actually a sheaf on $X$. Moreover $\Sh(X)$ is an abelian category. 
		
		\begin{proposition}\cite[Remark 2.2.5.]{KS90}
		\label{stalk_exact}
			A complex of sheaves $F\to G\to H$ in $\Sh(X)$ is exact if and only if for each $x\in X$ the sequence of vector spaces $F_x\to G_x\to H_x$ in $\Vect_\Bk$ is exact. In particular, the functor $(-)_x:\Sh(X)\to\Vect_\Bk$ is exact. 
		\end{proposition}
		\begin{proof}
			We first remark that for a morphism $\Gvf$ of sheaves, we notice $(\Ker\Gvf)_x\simeq\Ker(\Gvf_x)$ and $(\Coker\Gvf)_x\simeq\Coker(\Gvf_x)$ by Definition \ref{kernel and cokernel of sheaves}. Let $F\to G\to H$ be a complex of sheaves in $\Sh(X)$. There is a canonical morphism $\Im(F\to G)\to \Ker(G\to H)$. By Proposition \ref{stalk_iso}, it is an isomorphism if and only if $F_x\to G_x\to H_x$ is the exact sequence of $\Vect_\Bk$ for any $x\in X$. 
		\end{proof}
	
	\subsection{Fundamental operations}
	\label{fundamental operations}
		%Grothendieck six operations 
		%$(-)_Z, \GG_Z$
		Let $f:X\to Y$ be a continuous map. 
		
		\begin{definition}
		\label{pushforward and pullback}
			\begin{enumerate}[(i)]
				\item Let $F$ be a sheaf on $X$. The {\bf direct image} $f_*F$ of $F$ by $f$ is the sheaf on $Y$ defined by
				\begin{equation*}
					V\mapsto f_*F(V)=F(f^{-1}(V)). 
				\end{equation*}
				\item Let $G$ be a sheaf on $Y$. The {\bf inverse image} $f^{-1}G$ of $G$ by $f$ is the sheaf on $X$ associated to the presheaf
				\begin{equation*}
					U\mapsto \varinjlim_{V} G(V), 
				\end{equation*}
				where $V$ ranges through the family of open neighborhoods of $f(U)$ in $Y$. 
			\end{enumerate}
		\end{definition}
		
		\begin{remark}
			\begin{enumerate}[(i)]
				\item There are canonical morphism of functors
				\begin{equation}
				\label{canonical morphism for pushforward and pullback}
					f^{-1}\circ f_*\to \id_{\Sh(X)},\quad \id_{\Sh(Y)}\to f_*\circ f^{-1} . 
				\end{equation}
				Then $f^{-1}$ is a left adjoint to $f_*$ and $f_*$ is a right adjoint to $f^{-1}$. Namely, we have the following isomorphism
				\begin{equation}
				\label{adjoint_pushforward and pullback}
					\Hom_{\Sh(X)}(f^{-1}G, F)\simeq \Hom_{\Sh(Y)}(G, f_*F), 
				\end{equation}
				where $F\in\Sh(X)$ and $G\in\Sh(Y)$. 
				\item Let $G\in\Sh(Y)$ and let $x\in X$. Then $(f^{-1}G)_x\simeq G_{f(x)}$. 
				\item In addition, $f^{-1}:\Sh(Y)\to\Sh(X)$ is an exact functor but $f_*:\Sh(X)\to\Sh(Y)$ is only left exact. 
			\end{enumerate}
		\end{remark}
		
		\begin{definition}
		\label{G-section}
			For an open subset $U$ of $X$, let $\GG(U; -)$ denote the functor $\Sh(X)\to \Vect_\Bk, F\mapsto F(U)$. 
		\end{definition}
		For any open subset $U$ of $X$, the functor $\GG(U; -)$ is left exact. 
		
		\begin{example}\cite[Example 2.3.2.]{KS90}
			Let $F$ be a sheaf on $X$. For $a_X:X\to \{\pt\}$, where $\{\pt\}$ is the set with one element, one has
			\[
				\Bk_X\simeq a_X^{-1}\Bk_{\{\pt\}},\quad \GG(X; F)=\GG(\{\pt\};a_{X*}F)\simeq a_{X*}F. 
			\]
		\end{example}
		
		\begin{notation}
		Let $F$ be a sheaf on $X$ and $Z$ be a subset of $X$. For an inclusion map $j:Z\to X$, one sets $F|_Z:=j^{-1}F$ and $\GG(Z; F):=\GG(Z; j^{-1}F)$. 
		\end{notation}
		
		\begin{proposition}\cite[Proposition 2.5.1.]{KS90}
		\label{section_representation}
			Let $F\in\Sh(X)$. For a closed subset $Z\subset X$, there is a canonical isomorphism
			\[
				\Gy:\varinjlim_{U}\GG(U; F)\simeq \GG(Z; F), 
			\]
			where $U$ ranges through the family of open neighborhoods of $Z$ in $X$. 
		\end{proposition}
		\begin{proof}
			We shall prove that $\Gy$ is injective. If $s\in\GG(U; F)$ is zero in $\GG(Z; F)$, $s_x=0$ for any $x\in Z$. Then $s$ is zero on an open neighborhood of $Z$. Therefore $\Gy$ is an injective morphism. 
			
			We shall prove that $\Gy$ is surjective. Let $s\in\GG(Z; F)$. Then there exists an open covering $\{U_i\}_{i\in I}$ of $Z$ and $s_i\in\GG(U_i; F)$ such that $s|_{U_i\cap Z}=s_i|_{U_i\cap Z}$. Since $Z$ is a closed subset of the manifold $X$, $Z$ is paracompact. Then we may assume that $\{U_i\}$ is a locally finite covering of $Z$. We can find a family of open subsets $\{V_i\}_{i\in I}$ such that $\ol V_i\subset U_i$, $\{\ol V_i\}$ is locally finite and $\bigcup V_i\supset Z$. For $x\in X$, we set $I(x)=\{i\in I; x\in \ol V_i\}$ and define
			\[
				W=\2x\in\bigcup V_i; s_{i, x}=s_{j, x} \text{ for any } i, j\in I(x)\3. 
			\]
			Then $I(x)$ is a finite set and each $x$ has a neighborhood $W_x$ such that $I(y)\subset I(x)$ for any $y\in W_x$. Therefore $W$ is open and $W$ contains $Z$ by its construction. Since $s_i|_{W\cap V_i\cap V_j}=s_j|_{W\cap V_i\cap V_j}$, there exists $\widetilde s\in\GG(W; F)$ such that $\widetilde s|_{W\cap V_i}=s_i|_{W\cap V_i}$. Then $\widetilde s$ satisfies $\Gy(\widetilde s)=s$. 
		\end{proof}
		
		\begin{definition}
		\label{support}
			Let $F$ be a sheaf on $X$. 
			\begin{enumerate}[(i)]
				\item One defines the {\bf support of a sheaf} $F$ denoted by $\supp(F)$ as the complement of the union of open subsets $U\subset X$ such that $F|_U=0$. 	
				\item For an open subset $U$ of $X$ and a section $s\in\GG(U;F)$, one defines the {\bf support of a section} $s$ as the complement in $U$ of the union of the open sets $V\subset U$ such that $\Gr^F_{V, U}(s)=0$. 
			\end{enumerate}
		\end{definition}
		
		%Hom, テンソル
		\begin{definition}
		\label{internal hom and tensor}
			Let $F, G$ be sheaves on $X$. 
			\begin{enumerate}[(i)]
				\item One defines $\HOM(F, G)$ as the presheaf $U\mapsto \Hom_{\Sh(U)}(F|_U, G|_U)$. 
				\item One defines $F\otimes G$ as an associated sheaf to the presheaf $U\mapsto F(U)\otimes_\Bk G(U)$. 
			\end{enumerate}
		\end{definition}
		Since the presheaf $U\mapsto \Hom_{\Sh(U)}(F|_U, G|_U)$ satisfies the gluing condition, it is actually a sheaf on $X$. 
		\begin{remark}
			\begin{enumerate}[(i)]
				\item The bifunctor $(-)\otimes(-)$ is a left adjoint to $\HOM(-, -)$ and $\HOM(-, -)$ is a right adjoint to $(-)\otimes(-)$. Namely, we have the following isomorphism
				\[
					\Hom_{\Sh(X)}(F\otimes G, H)\simeq \Hom_{\Sh(X)}(F, \HOM(G, H)), 
				\]
				where $F, G, H\in \Sh(X)$. In addition, $(-)\otimes(-):\Sh(X)\times\Sh(X)\to\Sh(X)$ is an exact functor because $\Bk$ is a field but $\HOM:\Sh(X)^\circ\times\Sh(X)\to\Sh(X)$ is only left exact. 
				\item Let $F, G\in\Sh(X)$ and let $x\in X$. Then $(F\otimes G)_x\simeq F_x\otimes G_x$. 
			\end{enumerate}
		\end{remark}
		
		\begin{proposition}\cite[Proposition 2.3.5.]{KS90}
		\label{tensor_pullback}
			Let $F, G\in \Sh(Y)$. Then there is a canonical isomorphism
			\[
				f^{-1}F\otimes f^{-1}G\simeq f^{-1}(F\otimes G). 
			\]
		\end{proposition}
		\begin{proof}
			There is a canonical morphism $f^{-1}F\otimes f^{-1}G\to f^{-1}(F\otimes G)$ which is induced by $F(V)\otimes G(V)\to (F\otimes G)(V)$ for $V$ is an open subset of $Y$. For any $x\in X$, 
			\begin{align*}
				(f^{-1}F\otimes f^{-1}G)_x&\simeq (f^{-1}F)_x\otimes (f^{-1}G)_x\\
				&\simeq F_{f(x)}\otimes G_{f(x)}\\ 
				&\simeq (F\otimes G)_{f(x)}. 
			\end{align*}
			By Proposition \ref{stalk_iso}, we have the canonical isomorphism $f^{-1}F\otimes f^{-1}G\simeq f^{-1}(F\otimes G)$. 
		\end{proof}
		
		\begin{definition}
		\label{proper push}
			Let $F$ be a sheaf on $X$. 
			\begin{enumerate}[(i)]
				\item The {\bf proper direct image} $f_!F$ of $F$ by $f:X\to Y$ is the sheaf on $Y$ defined by
				\[
					V\mapsto \{s\in F(f^{-1}(V)); f:\supp(s)\to V {\rm\ is\ proper}\}. 
				\]
				\item One sets 
				\[
					\GG_c(X; F)=\{s\in \GG(X; F); \supp(s) {\rm\ is\ compact}\}. 
				\]
			\end{enumerate}
		\end{definition}
		By the above definition, one finds $\GG_c(X; F)\simeq a_{X!}F$ immediately.  
		
		\begin{proposition}\cite[Proposition 2.5.2.]{KS90}
		\label{proper stalk}
			Let $F$ be a sheaf on $X$ and $f:X\to Y$ be a continuous map. Then for $y\in Y$, the canonical morphism
			\[
				(f_! F)_y\to \GG_c(f^{-1}(y); F|_{f^{-1}(y)})
			\]
			is an isomorphism. 
		\end{proposition}
		\begin{proof}
			We first construct the morphism $\Ga:(f_! F)_y\to \GG_c(f^{-1}(y); F|_{f^{-1}(y)})$. 
			An element of $(f_! F)_y$ is represented by $t\in \GG(V; f_!F)$, where $V$ is an open neighborhood of $y\in Y$. Since $\GG(V; f_!F)=\{s\in F(f^{-1}(V)); f:\supp(s)\to V {\rm\ is\ proper}\}$ and $f^{-1}(V)$ is an open neighborhood of $f^{-1}(y)\subset X$, we can obtain an element of $\GG_c(f^{-1}(y); F|_{f^{-1}(y)})$ by restriction of $t$ on $f^{-1}(x)$. 
			
			We shall construct the inverse of $\Ga$. Let $s\in\GG_c(f^{-1}(y); F|_{f^{-1}(y)})$. We set $K=\supp(s)$. By Proposition \ref{section_representation}, there exists an open neighborhood $U$ of $K$ in $X$ and $t\in\GG(U; G)$ such that $t|_K=s|_K$. By shrinking $U$, we may assume $t|_{U\cap f^{-1}(y)}=s|_{U\cap f^{-1}(y)}$. Let $V$ be a relatively compact open neighborhood of $K$, with $\ol V\subset  U$. Since $y$ does not belong to $f(\ol V\cap \supp(t)\bs V)$, there exists an open neighborhood $W$ of $y$ such that $f^{-1}(W)\cap \ol V\cap \supp(t)\subset V$. We define $\widetilde s\in\GG(f^{-1}(W); F)$ by setting
			\[
				\widetilde s|_{f^{-1}(W)\bs(\supp(t)\cap \ol V)}=0,\quad 
				\widetilde s|_{f^{-1}(W)\cap V}=t|_{f^{-1}(W)\cap V}. 
			\]
			Since $\supp(\widetilde s)$ is contained in $f^{-1}(W)\cap \supp(t)\cap \ol V$, $f$ is proper on this set. Moreover we have $\widetilde s|_{f^{-1}(y)}=s$ and obtain the inverse of $\Ga$. 
		\end{proof}%あとでノート見て
		
		\begin{corollary}\cite[Remark 2.5.3.]{KS90}
		\label{push stalk}
			Let $F$ be a sheaf on $X$ and $f:X\to Y$ be a continuous map. Assume $f$ is proper on $\supp(F)$. Then for $y\in Y$, the canonical morphism
			\[
				(f_* F)_y\to \GG(f^{-1}(y); F|_{f^{-1}(y)})
			\]
			is an isomorphism. 
		\end{corollary}
		%\begin{proof}
			%%%%%%%%%%%%%
		%\end{proof}
		
		\begin{proposition}\cite[Proposition 2.5.11.]{KS90}
		\label{cartesian_commutative}
			For the following Cartesian square of finite dimensional manifolds
			\begin{equation}
			\label{cartesian_square}
				\xymatrix{
				X^\prime\ar[r]^{f^\prime}\ar[d]_{g^\prime}&Y^\prime\ar[d]^{g}\\
				X\ar[r]_{f}\ar@{}[ur]|{\square}&Y,
				}
			\end{equation}
			there is a canonical isomorphism of functors
			%\begin{enumerate}[(i)]
				%\item $
			\begin{equation}
			\label{cartesian_isomorphsim}
				g^{-1}\circ f_!\simeq f^{\prime}_!\circ g^{\prime-1}. 
			\end{equation}
				%$
				%\item $g^{-1}\circ Rf_!\simeq Rf^{\prime}_!\circ g^{\prime-1}$
				%\item $f^{!}\circ Rg_*\simeq Rg^{\prime}_*\circ f^{\prime!}$
			%\end{enumerate}
		\end{proposition}
		
		\begin{proof}
			First we shall construct a canonical morphism
			\begin{equation}
			\label{cartesian_morphism}
				f_!\circ g^\prime_*\to g_*\circ f^\prime_!. 
			\end{equation}
			Let $F\in\Sh(X^\prime)$ and let $V$ be an open subset of $Y$. Consider the following morphisms
			\begin{equation}
				\xymatrix{
				\GG(V; f_!\circ g^\prime_* F)\ar[r]\ar[d]&\GG(f^{-1}(V); g^\prime_* F)\ar[r]^-{\simeq}&\GG(g^{\prime-1}(f^{-1}(V)); F)\ar[d]^{\simeq}\\
				\GG(V; g_*\circ f^\prime_! F)\ar[r]_-{\simeq}&\GG(g^{-1}(V); f^\prime_!F)\ar[r]&\GG(f^{\prime-1}(g^{-1}(V)); F). 
				}
			\end{equation}
			A section $s\in\GG(V; f_!\circ g^\prime_*F)$ is given by a section $t\in\GG(f^{-1}(V); g^\prime_* F)$ such that $f:Z\to V$ is proper with $Z:=\supp(t)$. For the section $t$, we have a section $u\in\GG(g^{\prime-1}(f^{-1}(V)); F)\simeq\GG(f^{\prime-1}(g^{-1}(V)); F)$ such that $Z=g^\prime(\supp(u))$. 
			By the following Cartesian square
			\[
				\xymatrix{
				g^{\prime-1}(Z)\ar[r]^{f^\prime}\ar[d]_{g^\prime}&g^{-1}(V)\ar[d]^{g}\\
				Z\ar[r]_{f}\ar@{}[ur]|{\square}&V, 
				}
			\]
			$f^\prime:g^{\prime-1}(Z)\to g^{-1}(V)$ is proper. 
			Therefore $u$ defines a section of $g_*\circ f^\prime_! F$ on $V$ and this assignment gives the morphism (\ref{cartesian_morphism}). 
			
			We consider the composition of (\ref{canonical morphism for pushforward and pullback}) and (\ref{cartesian_morphism}):
			\begin{equation}
			\label{cartesian_morphism2}
				f_!\to f_!\circ g^{\prime}_*\circ g^{\prime-1}\to g_*\circ f^\prime_!\circ g^{\prime-1}. 
			\end{equation}
			By the isomorphism (\ref{adjoint_pushforward and pullback}), for $F\in\Sh(X)$, we have
			\[
				\Hom_{\Sh(Y^\prime)}(g^{-1}\circ f_!F, f^{\prime}_!\circ g^{\prime-1}F)\simeq \Hom_{\Sh(Y)}(f_!F, g_*\circ f^\prime_!\circ g^{\prime-1}F).  
			\]
			We get the morphism (\ref{cartesian_isomorphsim}) which is the image of the morphism (\ref{cartesian_morphism2}) by the above isomorphism. 
			
			To prove this is an isomorphism, we calculate the stalk at $y\in Y^\prime$: 
			\begin{align*}
				(g^{-1}\circ f_!F)_y&\simeq (f_!F)_{g(y)}\\
				&\simeq \GG_c(f^{-1}(g(y)); F|_{f^{-1}(g(y))})\\
				&\simeq \GG_c(f^{\prime-1}(y); g^{\prime-1}F|_{f^{\prime-1}(y)})\\
				&\simeq (f^{\prime}_!\circ g^{\prime-1}F)_y. 
			\end{align*}
			The third isomorphism is induced by the homeomorphism $g^{\prime}:f^{\prime-1}(y)\to f^{-1}(g(y))$. 
		\end{proof}
		
		\begin{proposition}\cite[Proposition 2.5.12.]{KS90}
		\label{tensor_vector space}
			Let $L$ be a vector space and let $F$ be a sheaf on $X$. Then there is a natural isomorphism
			\begin{equation}
				\GG_c(X; F)\otimes L\simeq \GG_c(X; F\otimes L_X). 
			\end{equation}
			Here $L_X$ is a constant sheaf on $X$ defined by $a_X^{-1}L$. 
		\end{proposition}
		
		\begin{proposition}\cite[Proposition 2.5.13.]{KS90}
		\label{projection_formula}
			For $F\in\Sh(X)$ and $G\in\Sh(Y)$, there is a natural  isomorphism
			\begin{equation}
			%\label{projection_formula_morphism}
				f_!F\otimes G\to f_!(F\otimes f^{-1}G). 
			\end{equation}
		\end{proposition}
		\begin{proof}
			Note first that we have the canonical morphism $f_*G\otimes F\to f_*(G\otimes f^{-1}F)$ as the image of identity in $\Hom_{\Sh(Y)}(G\otimes f^{-1}F)$ by the composition of the following morphisms
			\begin{align*}
				\Hom_{\Sh(Y)}(G\otimes f^{-1}F, G\otimes f^{-1}F)&\to \Hom_{\Sh(Y)}(f^{-1}f_*G\otimes f^{-1}F, G\otimes f^{-1}F)&\\
				&\simeq\Hom_{\Sh(X)}(f_*G\otimes F, f_*(G\otimes f^{-1}F)). &
			\end{align*}
			The above isomorphism is obtained by Proposition \ref{tensor_pullback}. Then the morphism (\ref{projection_formula_morphism}) is induced by the canonical morphism $f_*G\otimes F\to f_*(G\otimes f^{-1}F)$. 
			
			In order to prove (\ref{projection_formula_morphism}) is isomorphism, let us choose $y\in Y$ and apply Lemma \ref{proper stalk}. We have
			\begin{align*}
				(f_!F\otimes G)_y&\simeq\GG_c(f^{-1}(y); G\otimes f^{-1}F|_{f^{-1}(y)})&\\
				&\simeq\GG_c(f^{-1}(y); G|_{f^{-1}(y)}\otimes (F_y)_X)&\\
				&\simeq\GG_c(f^{-1}(y); G|_{f^{-1}(y)})\otimes (F_y)&\\
				&\simeq(f_!G)_y\otimes F_y&\\
				&\simeq f_!(G\otimes f^{-1}F). &
			\end{align*}
			Since $F_y$ is a $\Bk$-vector space, the third isomorphism is induced by Proposition \ref{tensor_vector space}. 
			Then we obtain the natural  isomorphism (\ref{projection_formula_morphism}) by Proposition \ref{stalk_iso}. 
		\end{proof}
		
		%切り落とし, がんま切り落とし
		
		\begin{definition}
		\label{cut-off_def}
			Let $F$ be a sheaf on $X$. 
			\begin{enumerate}[(i)]
				\item For a closed subset $A$ of $X$ and the inclusion $j:A\hookrightarrow X$, one defines 
				\begin{equation}
					F_A:=j_*j^{-1}F. 
				\end{equation}
				Then we have a natural morphism $F\to F_A$. 
				\item For an open subset $W$ of $X$, one defines
				\begin{equation}
					F_W:= \Ker(F\to F_{X\bs W}). 
				\end{equation}
				\item For a locally closed subset $Z$ of $X$ written by $Z=A\cap W$, where $A$ is a closed subset of $X$ and $U$ is an open subset of $X$. Then one defines
				\begin{equation}
					F_Z:=(F_A)_W. 
				\end{equation}
			\end{enumerate}
		\end{definition}
		
		For a locally closed subset $Z$ of $X$, we simply write $\Bk_Z$ instead of $(\Bk_X)_Z$ in $\Sh(X)$, where $\Bk_X$ is a constant sheaf on $X$. 
		
		\begin{remark}
			Let $U$ be an open subset of $X$. 
			\begin{enumerate}[(i)]
				\item For a locally closed subset $Z$ of $X$, one has
				\begin{equation}
				\label{cut-off_restriction}
					\begin{cases}
						F_Z|_Z\simeq F|_Z, &\\
						F_Z|_{X\bs Z}=0. &
					\end{cases}
				\end{equation}
				In particular, $(F_Z)_x\simeq F_x$ if $x\in Z$ and $(F_Z)_x=0$ if $x\not\in Z$. Moreover if $F^\prime\in\Sh(X)$ satisfies $F^\prime|_Z\simeq F|_Z$ and $F^\prime|_{X\bs Z}\simeq F|_{X\bs Z}$, there exists an isomorphism $F_Z\simeq F^\prime$. 
				\item For a closed subset $A$ of $X$, the sheaf $F_A$ is the associated sheaf of the following presheaf
				\begin{equation}
				%\label{cut-off_restriction}
					U\mapsto
					\begin{cases}
						F(U)&\text{ if }U\cap A\neq\emptyset, \\
						0&\ow. 
					\end{cases}
				\end{equation}
				\item For an open subset $W$ of $X$, the sheaf $F_W$ satisfies
				\begin{equation}
				F_W(U)\simeq
					\begin{cases}
						F(U)&\text{ if }U\subset W, \\
						0&\ow. 
					\end{cases}
				\end{equation}
			\end{enumerate}
		\end{remark}
		
		\begin{proposition}\cite[Proposition 2.5.4.]{KS90}
		\label{i!}
			Let $Z$ be a locally closed of $X$ and let $i:Z\hookrightarrow X$ be the inclusion. 
			\begin{enumerate}[(i)]
				\item The functor $i_!$ is exact. 
				\item Let $F$ be a sheaf on $X$. Then
				\[
					F_Z\simeq i_!i^{-1}F. 
				\]
			\end{enumerate}
		\end{proposition}
		\begin{proof}
			(i) If $x\in Z$, $(i_!F)_x\simeq F_x$ and if $x\not\in Z$, $(i_!F)_x=0$. Therefore $i_!$ is exact. 
			
			(ii) We have $(i_!i^{-1}F)|_{Z}\simeq F|_Z$ and $(i_!i^{-1}F)|_{X\bs Z}=0$. Therefore we obtain $F_Z\simeq i_!i^{-1}F$. 
		\end{proof}
		
		\begin{proposition}\cite[Proposition 2.3.6.(iii)]{KS90}
		\label{cut-off_commutative}
			Let $Z$ and $Z^\prime$ be two locally closed subsets of $X$. Let $F$ be a sheaf on $X$. There is a natural isomorphism
			\[
				F_{Z\cap Z^\prime}\simeq (F_Z)_{Z^\prime}. 
			\]
		\end{proposition}
		\begin{proof}
			Consider the following Cartesian square
			\begin{equation*}
				\xymatrix{
				Z\cap Z^\prime\ar[r]^{j}\ar[d]_{j^\prime}&Z^\prime\ar[d]^{i^\prime}\\
				Z\ar[r]_{i}\ar@{}[ur]|{\square}&X, 
				}
			\end{equation*}
			where four morphisms are inclusions. By Proposition \ref{cartesian_commutative} and Proposition \ref{i!}, we have
			\begin{align*}
				(F_Z)_{Z^\prime}\simeq i^\prime_!i^{\prime-1}i_!i^{-1}F
				\simeq i^\prime_!j_!j^{\prime-1}i^{-1}F
				\simeq F_{Z\cap Z^\prime}. 
			\end{align*}
		\end{proof}
		
		\begin{proposition}\cite[Remark 2.3.11.]{KS90}
		\label{cut-off_pullback}
			Let $Z$ be a locally closed subset of $Y$. There is a natural isomorphism
			\[
				f^{-1}F_Z\simeq (f^{-1}F)_{f^{-1}(Z)}. 
			\]
		\end{proposition}
		\begin{proof}
			Let $i:f^{-1}(Z)\hookrightarrow X$ and $j:Z\hookrightarrow Y$ denote the inclusion. Consider the following Cartesian square
			\begin{equation*}
				\xymatrix{
				f^{-1}(Z)\ar[r]^{i}\ar[d]_{f|_{f^{-1}(Z)}}&X\ar[d]^{f}\\
				Z\ar[r]_{j}\ar@{}[ur]|{\square}&Y.  
				}
			\end{equation*}
			By Proposition \ref{cartesian_commutative}, we have
			\begin{align*}
				f^{-1}(F_Z)&= f^{-1}j_!j^{-1}F
				\simeq i_!f|_{f^{-1}(Z)}^{-1}j^{-1}F
				\simeq i_!i^{-1}f^{-1}F
				\simeq (f^{-1}F)_{f^{-1}(Z)}. 
			\end{align*}
		\end{proof}
		
		\begin{proposition}\cite[Proposition 2.3.10.]{KS90}
		\label{cut-off_tensor}
			Let $Z$ be a locally closed subset of $X$. For a sheaf $F$ on $X$, there is a natural isomorphism
			\[
				F\otimes \Bk_{Z}\simeq F_Z. 
			\]
		\end{proposition}
		\begin{proof}
			By Proposition \ref{tensor_pullback}, we have $(F\otimes \Bk_{Z})|_{Z}\simeq F|_Z$ and $(F\otimes \Bk_{Z})|_{X\bs Z}=0$. Therefore we obtain $F\otimes \Bk_{Z}\simeq F_Z$. 
		\end{proof}
		
		\begin{proposition}\cite[Proposition 2.3.6.(v)]{KS90}
		\label{short exact}
			%\begin{enumerate}[(i)]
				%\item 
				Let $Z$ be a locally closed subset of $X$ and $Z^\prime$ be a closed subset of $Z$. For $F\in\Sh(X)$, one obtains the following exact sequence
				\begin{equation*}
					0\to F_{Z\bs Z^\prime}\to F_Z\to F_{Z^\prime}\to0. 
				\end{equation*}
				%\item For a locally closed subset $Z\subset X$, $(-)_Z:\Sh(X)\to \Sh(X)$ is exact. 
				%\item For $F\in\Sh(X)$, if $F^\prime\in\Sh(X)$ satisfies (\ref{cut-off_restriction}), $F_Z\simeq F^\prime$. 
			%\end{enumerate}
		\end{proposition}
		\begin{proof}
			This follows from Proposition \ref{stalk_iso}. 
		\end{proof}

		\begin{definition}
		\label{relative homology_def}
			Let $F$ be a sheaf on $X$. For a locally closed subset $Z$ of $X$, let $j:Z\hookrightarrow X$ denote the inclusion. 
			\begin{enumerate}[(i)]
				%\item We set $F_Z:=j_!j^{-1}F\in\Sh(X)$. 
				\item For any open subset $U\subset X$, one defines $\GG_Z(U; F)$ as the vector space $\Ker(F(V)\to F(V\bs Z))$, where $V$ is an open subset of $U$ containing $Z$ as a closed subset. 
				\item One defines $\GG_Z(F)$ as a presheaf $U\mapsto \GG_{Z\cap U}(U; F)$ on $X$. 
			\end{enumerate}	
		\end{definition}
		
		\begin{remark}
			Consider the above setting. 
			\begin{enumerate}[(i)]
				\item For any pair of open subsets $V\subset U$ containing $Z$ as a closed subset, the canonical morphism $\GG_Z(U; F)\to \GG_Z(V; F)$ is an isomorphism. Therefore $\GG_Z(U; F)$ is independent of the choice of $V$. 
				\item Since the presheaf $U\mapsto \GG_{Z\cap U}(U; F)$ satisfies the gluing condition, it is actually a sheaf on $X$. 
				\item The functor $(-)_Z$ is a left adjoint to $\GG_Z(-)$ and $\GG_Z$ is a right adjoint to $(-)_Z$. In addition, $(-)_Z:\Sh(X)\to\Sh(X)$ is an exact functor but $\GG_Z:\Sh(X)\to\Sh(X)$ is only left exact. 
			\end{enumerate}
		\end{remark}

	\subsection{Derived category}
	\label{derived category}\phantom{a}
		We  shall henceforth consider the derive category of the abelian category. We denote by $D(X)$ the derived category of $\Sh(X)$ and by $D^b(X)$ the bounded derived category of that. The category $D^b(X)$ is a full triangulated subcategory of $D(X)$. If there is no risk of confusion, an object $F$ of $D^b(X)$ is simply called a sheaf instead of a complex of sheaves. 
		
		In order to define the derived functors of the functors appeared in Section.\ref{fundamental operations}, we consider that $Z\subset X$ is a locally closed subset and $f:X\to Y$ is a continuous map of finite dimensional manifolds. Since $\Sh(X)$ has enough injectives, we can obtain derived functors from all left exact functors:
		\[
			 R\GG(Z; -), R\GG_{Z}(X; -), R\GG_Z(-), Rf_*, \RHom(-, -), \RHOM(-, -), Rf_!, R\GG_c(X; -). 
		\]
		Moreover, since the functors $(-)_Z, f^{-1}, (-)\otimes(-)$ are exact, they are extended to derived categories. 
		
		We show some propositions for derived categories of sheaves. 
		
		\begin{proposition}\cite[Proposition 2.6.7.]{KS90}
		\label{cartesian_commutative_derived.ver}
			Consider the Cartesian square (\ref{cartesian_square}). There is a canonical isomorphism of functors
			\[
				g^{-1}\circ Rf_!\simeq Rf^{\prime}_!\circ g^{\prime-1}. 
			\]
		\end{proposition}
		
		This proposition is obtained by Proposition \ref{cartesian_square}. 
		
		\begin{proposition}\cite[Proposition 2.6.6]{KS90}
		\label{projection_formula_derived.ver}
			For $F\in D^b(X)$ and $G\in D^b(Y)$, there is a natural  isomorphism
			\begin{equation}
			\label{projection_formula_morphism}
				Rf_!F\otimes G\to Rf_!(F\otimes f^{-1}G). 
			\end{equation}
		\end{proposition}
		
		This proposition is obtained by Proposition \ref{projection_formula}. 
		
		\begin{proposition}
		\label{distinguish triangle}
			Let $Z$ be a locally closed subset of $X$ and $Z^\prime$ be a closed subset of $Z$. For $F\in D^b(X)$, there is the following distinguish triangle
			\[
				F_{Z\bs Z^\prime}\to F_Z\to F_{Z^\prime}\overset{+1}{\to}
			\]
			in $D^b(X)$. 
		\end{proposition}
		
		This distinguish triangle is showed in \cite[(2.6.33.)]{KS90} and obtained by Proposition \ref{short exact}. 
		
		For considering the adjoint functor of the functor induced by the composition bifunctors discussed in Section.\ref{composition}, we need to define the right adjoint functor of $Rf_!$. 
		
		\begin{theorem}\cite[Theorem 3.1.5.]{KS90}
		\label{exceptional pullback}
			Let $f:X\to Y$ be a continuous map of finite dimensional manifolds. Then there exists a functor $f^!:D^b(Y)\to D^b(X)$ of triangulated categories and an isomorphism
			\[
				\Hom_{D^b(Y)}(Rf_!F, G)\simeq \Hom_{D^b(X)}(F, f^!G), 
			\]
			where $F\in D^b(X)$ and $G\in D^b(Y)$. 
		\end{theorem}
		
	\subsection{Composition}
	\label{composition}\phantom{a}
		Let $X_i\ (i=1, 2, 3, 4)$ be finite dimensional manifolds. We set $X_{ij}=X_i\times X_j$ and $X_{123}=X_1\times X_2\times X_3$. We denote by $p_{ij}:X_{123}\to X_{ij}$ the projection of the corresponding indexes. 
		\begin{definition}
		\label{composition_def}
			For $K_{ij}\in D^b(X_{ij})$, the composition bifunctor $D^b(X_{12})\times D^b(X_{23})\to D^b(X_{13})$ is defined as follows: 
			\begin{equation}
				K_{12}\circ_{X_2}K_{23}:=Rp_{13!}(p_{12}^{-1}K_{12}\otimes p_{23}^{-1}K_{23}). 
			\end{equation}
			We usually write $\circ_{2}$ instead of $\circ_{X_2}$. If there is no risk of confusion, we simply write $\circ$. 
		\end{definition}
		\begin{example}
		\label{composition functor}
			If $X_3=\pt$, one modifies $p_{i3}:X_{123}\to X_{i3}$ to $p_i:X_{12}\to X_i$ and $p_{12}:X_{123}\to X_{12}$ to $\id_{X_{12}}$. One sets the functor $\GF_K:D^b(X_2)\to D^b(X_1)$ for $K\in D^b(X_{12})$ as follows:
			\begin{equation}
				\GF_K(F):=Rp_{1!}(K\otimes p_2^{-1}F), 
			\end{equation}
			where $F\in D^b(X_2)$. 
			$\GF_K$ has a right adjoint functor $\GY_K: D^b(X_1)\to D^b(X_2)$ as follows:
			\begin{equation}
				\GY_K(F):= Rp_{2*}R\mathcal{H}om(K, p_1^!F), 
			\end{equation}
			where $F\in D^b(X_1)$. 
		\end{example}
		\begin{proposition}
		\label{associativity of composition}
			For $K_{ij}\in D^b(X_{ij})$, there is a canonical isomorphism
			\begin{equation}
				(K_{12}\circ_2 K_{23})\circ_3 K_{34}\simeq K_{12}\circ_2 (K_{23}\circ_3 K_{34}). 
			\end{equation}
		\end{proposition}
		
		Let $M$ be a smooth manifold. We denote by $\widetilde p_{ij}:M\times X_{123}\to M\times X_{ij}$ the projection. We define the tilde composition bifunctor which is similar to the composition bifunctor. The tilde composition bifunctor is necessary for the proof of Theorem \ref{re:main theorem}. 
		\begin{definition}
		\label{tilde composition}
			For $K_{ij}\in D^b(M\times X_{ij})$, the tilde composition bifunctor $D^b(M\times X_{12})\times D^b(M\times X_{23})\to D^b(M\times X_{13})$ is defined as follow:
			\begin{equation}
				K_{12}\tcp_{X_2}K_{23}:=R\widetilde p_{13!}(\widetilde p_{12}^{\ -1}K_{12}\otimes \widetilde p_{23}^{\ -1}K_{23}). 
			\end{equation}
			We usually write $\tcp_{\!2}$ instead of $\tcp_{\!X_2}$. If there is no risk of confusion, we simply write $\tcp$. 
		\end{definition}
		\begin{proposition}
		\label{associativity of tilde composition}
			For $K_{12}\in D^b(M\times X_{12}), K_{23}\in D^b(M\times X_{23}), K_{34}\in D^b(X_{34})$, there is a canonical isomorphism: 
			\begin{equation}
				(K_{12}\tcp_{\!2} K_{23})\circ_3 K_{34}\simeq K_{12}\tcp_{\!2} (K_{23}\circ_3 K_{34}). 
			\end{equation}
		\end{proposition}
		\begin{proof}
			We define the following projections
			\[
				\xymatrix{
					&&M\times X_{1234}\ar[dl]_{\tq_{123}}\ar[dr]^{\tp_{134}}&&\\
					&M\times X_{123}\ar[dl]_{\tp_{12}}\ar[d]_{\tp_{23}}\ar[dr]_{\tp_{13}}&\ar@{}@<0ex>|{\square}&M\times X_{134}\ar[dl]^{\tq_{13}}\ar[d]^{\tq_{14}}\ar[dr]^{\tq_{34}}&\\
					M\times X_{12}&M\times X_{23}&M\times X_{13}&M\times X_{14}&X_{34}, 
				}
			\]
			where the above square is  the Cartesian square. By Proposition \ref{cartesian_commutative_derived.ver}, Proposition\ref{tensor_pullback} and Proposition \ref{projection_formula_derived.ver}, 
			\begin{align*}
				(K_{12}\tcp_{\!2} K_{23})\circ_3 K_{34}&=R\tq_{14!}(\tq_{13}^{\ -1}(K_{12}\tcp_{\!2} K_{23})\otimes\tq_{34}^{\ -1}K_{34})\\
				&=R\tq_{14!}(\tq_{13}^{\ -1}R\tp_{13!}(\tp_{12}^{\ -1}K_{12}\otimes \tp_{23}^{\ -1}K_{23})\otimes\tq_{34}^{\ -1}K_{34})\\
				&\simeq R\tq_{14!}(R\tp_{134!}\tq_{123}^{\ -1}(\tp_{12}^{\ -1}K_{12}\otimes \tp_{23}^{\ -1}K_{23})\otimes\tq_{34}^{\ -1}K_{34})\\
				&\simeq R\tq_{14!}(R\tp_{134!}(\tq_{123}^{\ -1}\tp_{12}^{\ -1}K_{12}\otimes \tq_{123}^{\ -1}\tp_{23}^{\ -1}K_{23})\otimes\tq_{34}^{\ -1}K_{34})\\
				&\simeq R\tq_{14!}R\tp_{134!}(\tq_{123}^{\ -1}\tp_{12}^{\ -1}K_{12}\otimes \tq_{123}^{\ -1}\tp_{23}^{\ -1}K_{23}\otimes \tp_{134}^{\ -1}\tq_{34}^{\ -1}K_{34}). 
			\end{align*}
			We denote the projections by $\tp^{\ \prime}_{12}:M\times X_{1234}\to M\times X_{12}, \tp^{\ \prime}_{23}:M\times X_{1234}\to M\times X_{23}$ and $\tp^{\ \prime}_{34}:M\times X_{1234}\to X_{34}$. Then we have the following isomorphism
			\[
				(K_{12}\tcp_{\!2} K_{23})\circ_3 K_{34}\simeq R\tp^{\ \prime}_{14!}(\tp^{\ \prime-1}_{12}K_{12}\otimes\tp^{\ \prime-1}_{23}K_{23}\otimes\tp^{\ \prime-1}_{34}K_{34}). 
			\]
			We can similarly prove an isomorphism
			\[
				K_{12}\tcp_{\!2} (K_{23}\circ_3 K_{34})\simeq R\tp^{\ \prime}_{14!}(\tp^{\ \prime-1}_{12}K_{12}\otimes\tp^{\ \prime-1}_{23}K_{23}\otimes\tp^{\ \prime-1}_{34}K_{34}). 
			\]
		\end{proof}

\section{Quantized contact transform}
\label{QCT}\phantom{a}
	In this chapter, we will show the quantized contact transform and the Radon transform. In order to define the quantize contact transform, we introduce some concepts in the sheaf theory: cohomologically constructible sheaves, microsupports and $\mu hom$. See Kashiwara-Schapira \cite{KS90} for details and proofs.  

	\subsection{Cohomologically constructible sheaves}
	\label{cohomologically constructible sheaves}\phantom{a}
		Let $\Ccal$ be a category.  We denote by $\Set$ the category of set and by $\Ccal^\vee$ the category of contravariant functor from $\Ccal$ to $\Set$. We regard $\Ccal$ as a full subcategory of $\Ccal^\vee$ by the functor $X\mapsto \Hom_{\Ccal}(-, X)$. 
	
		In this section, we show ind-object and pro-object in $\Ccal^\vee$ at first. 
	
		\begin{definition}
		\label{filtrant}
			A category $\Ical$ is called {\bf filtrant} if it is non-empty and satisfies the following conditions. 
			\begin{enumerate}[(i)]
				\item For $i, j\in\Ical$, there exist $k\in\Ical$ and morphisms $i\to k$ and $j\to k$. 
				\item For two morphisms $f, g\in \Hom_{\Ical}(i, j)$, there exists a morphism $h:j\to k$ such that $h\circ f=h\circ g$.
			\end{enumerate} 
		\end{definition}
		
		\begin{definition}
		\label{ind-pro-object}
			Let $\Ical$ and $\Ccal$ be two categories, $\Ical$ being filtrant. 
			\begin{enumerate}[(i)]
				\item For an inductive system $F$ in $\Ccal$ indexed by $\Ical$, let $\indlim{\Ical} F$ denote the functor from the opposite category $\Ccal^\circ$ to $\Set$:
				\[
					X\mapsto \varinjlim_{i\in\Ical} \Hom_{\Ccal}(X, F(i)). 
				\]
				\item For a projective system $F$ in $\Ccal$ indexed by $\Ical$, let $\prolim{\Ical} F$ denote the functor from $\Ccal$ to $\Set$:
				\[
					X\mapsto \varinjlim_{i\in\Ical} \Hom_{\Ccal}(F(i), X). 
				\]
			\end{enumerate}
			
			A functor from $\Ccal^\circ$ (resp.\,$\Ccal$) to $\Set$ is called an ind-object (resp.\,a pro-object) if it is isomorphic to $\indlim{\Ical} F$ (resp.\,$\prolim{\Ical} F$) for an inductive (resp.\,projective) system $F$ in $\Ccal$ indexed by $\Ical$. 
		\end{definition}
	
		\begin{remark}
			We may regard $\indlim{\Ical} F$ (resp.\,$\prolim{\Ical} F$) as an object in $\Ccal$ if it is representable. 
		\end{remark}
		
		\begin{example}\cite[Example 1.11.5.]{KS90}
			Let $(I, \leq)$ be an ordered set. We define a category $\Ical$ as follows:
			\begin{align*}
				\Ob(\Ical)&=I, \\
				\Hom_{\Ical}(i, j)&=
				\begin{cases}
					\pt&i\leq j, \\
					\emptyset&\ow. 
				\end{cases}
			\end{align*}
			If $I$ is a directed ordered set, the category $\Ical$ is filtrant. 
		\end{example}
	
		\begin{definition}
		\label{cohomologically constructible}
			An object $F$ of $D^b(X)$ is called {\bf cohomologically constructible}, if for any $x\in X$, the following conditions are satisfied. 
			\begin{enumerate}[(i)]
				\item Two objects $\indlim{x\in U}R\GG(U; F)$ and $\prolim{x\in U}R\GG_c(U; F)$ are representable ($U$ ranges through the family of open neighborhoods of $x$). 
				\item Two morphisms $\indlim{x\in U}R\GG(U; F)\to F_x$ and $R\GG_{\{x\}}(X; F)\to\prolim{x\in U}R\GG_c(U; F)$ are isomorphisms. 
				\item For any $i\in\Zbb$, $H^i(F_x)$ and $H^i(R\GG_{\{x\}}(X; F))$ are finite dimensional vector spaces. 
			\end{enumerate}
		\end{definition}

	\subsection{Microsupport}
	\label{microsupport}\phantom{a}
	%$SS(F)$
		We denote by $(x; \xi)$ a local homogeneous coordinate system on the cotangent bundle $T^*X$ of $X$. The zero section of cotangent bundle is denoted by $0_X$. We set $\dot T^*X:=T^*X\bs 0_X$. We show the definition of the {\bf microsupport} $\SS(F)\subset T^*X$\cite[Definition 5.1.2.]{KS90}. 
		
		\begin{definition}
			Let $F\in D^b(X)$ and let $p\in T^*X$. One says that $p\notin \SS (F)$ if there exists an open neighborhood $U$ of $p$ such that for any $x_0\in X$ and any real $C^1$-function $\Gvf$ on $X$ defined in a neighborhood of $x_0$ with $(x_0; d\Gvf(x_0))\in U$, one has $R\GG_{\{x; \Gvf(x)\geq \Gvf(x_0)\}}(F)_{x_0}\simeq 0$. 
		\end{definition}
		
		%localization
		
		Let $\GO$ be a subset of $T^*X$ and let $V=T^*X\bs \GO$. The full subcategory $D^b_{V}(X)$ of $D^b(X)$ consisting of sheaves $F$ such that $\SS(F)\subset V$ is triangulated. One sets 
		
		\begin{equation}
			D^b(X; \GO):=D^b(X)/D^b_{V}(X)
		\end{equation}
		as the localization of $D^b(X)$ of $D^b_{V}(X)$(see \cite[\S 6.1.]{KS90}). For $p\in T^*X$, if $\GO=\{p \}$, we write $D^b(X; p)$ instead of $D^b(X; \{p\})$. 
		
		\begin{example}%\textcolor{orange}{(Sect\ref{dlc})}
		\label{MVgamma}
			Let $M$ be a smooth manifold and $M\times \Vbb\to M$ be a trivial vector bundle. For a closed cone $\Gg\subset\Vbb$, we write its anti polar cone 
			\begin{equation}
				\Gg^{\circ a}:=\{\Gn\in\Vbb^*; \langle v, \Gn\rangle\leq 0 \text{\ for any\ } v\in\Gg\}
			\end{equation}
			as the cone of $\Vbb^*$. 
			We set $T^{*, \Gg}(M\times\Vbb):=T^*M\times \Vbb\times (\Vbb^*\bs\Gg^{\circ a})$. Especially, if $\Vbb=\Rbb$ and $\Gg=[0, \infty)$, we write $T^{*, +}(M\times \Rbb)$ instead of $T^{*, [0, \infty)}(M\times\Rbb)$. 
		\end{example}
		
		Let $M$ be a closed submanifold of $X$. The normal bundle $T_MX$ to $M$ in $X$ is defined as the cokernel of the morphism $TM\to M\times_{X}TX$. 
		Similarly, the conormal bundle $T^*_MX$ to $M$ in $X$ is defined as the kernel of the morphism $M\times_{X}T^*X\to T^*M$. 
		
		\begin{proposition}\cite[Proposition 6.6.1.(ii)]{KS90}
		\label{microsupport_prop}
			Let $p\in T^*_MX$ and let $F\in D^b(X)$. Assume $\SS(F)\subset T^*_MX$ in a neighborhood of $p$. Then there exists $L\in D^b(\Vect_\Bk)$ such that $F\simeq L_M$ in $D^b(X; p)$. 
		\end{proposition}
	
	\subsection{Microlocalization}
	\label{muhom}\phantom{a}
		Let $X$ be an $n$-dimensional manifold. Let $M$ be a closed submanifold of codimension $l$. 
		We shall construct the {\bf normal deformation} $\tX_M$ of $M$ in $X$ and two maps $p:\tX_M\to X$ and $t:\tX_M\to \Rbb$ such that
		\[
			\begin{cases}
				p^{-1}(X\bs M)&\text{is isomorphic to } (X\bs M)\times(\Rbb\bs\{0\}),\\
				t^{-1}(\Rbb\bs\{0\})&\text{is isomorphic to }X\times(\Rbb\bs\{0\}),\\
				t^{-1}(0)&\text{is isomorphic to }T_MX. 
			\end{cases}
		\]
		
		\ul{construction of $\tX_M$}
		
		We take an open covering $X=\bigcup_{i}U_i$ and open embeddings $\Gvf_i:U_i\to\Rbb^n$ such that $U_i\cap M=\Gvf_i^{-1}(\{0\}^l\times\Rbb^{n-l})$. Using a coordinate $x=(x^\prime, x^{\prime\prime})\in \Rbb^l\times\Rbb^{n-l}$, we define $V_i$ as
		\[
			V_i:=\{(x, t)\in\Rbb^n\times\Rbb; (tx^\prime, x^{\prime\prime})\in\Gvf_i(U_i)\}. 
		\]
		We denote by $t_{V_i}:V_i\to\Rbb$ the projection $(x, t)\mapsto t$, and by $p_{V_i}:V_i\to U_i$ the morphism $(x, t)\mapsto \Gvf_i^{-1}(tx^\prime, x^{\prime\prime })$. We define the map $\Gy_{ji}:V_i\times_{U_i}(U_i\cap U_j)\to\Rbb^n$ by setting $\Gy_{ji}(x, t)=(\Gy_{ji}^\prime, \Gy_{ji}^{\prime\prime}(x, t))$ with $(t\Gy_{ji}^\prime(x, t), \Gy_{ji}^{\prime\prime}(x, t))=\Gvf_j\Gvf_i^{-1}(tx^\prime, x^{\prime\prime})$. We note that the first $l$ components of $\Gvf_j\Gvf_i^{-1}(tx^\prime, x^{\prime\prime})$ vanish when $t=0$. 
		
		We set
		\[
			\tX_M:=\8\bigsqcup_{i}V_i\9/ \sim, 
		\]
		where $\sim$ is an equivalence relation which identifies $(x_i, t_i)\in V_i$ and $(x_j, t_j)\in V_j$ if $t_i=t_j$ and $x_j=\Gy_{ji}(x_i, t_i)$. We can immediately check $\tX_M$ is a manifold and two maps $p:\tX_M\to X$ and $t:\tX_M\to\Rbb$ defined by $p|_{V_i}=p_{V_i}$ and $t|_{V_i}=t_{V_i}$ respectively. 
		
		We introduce some maps related to the normal deformation of $M$ in $X$. We denote by $\GO$ the open subset of $\tX_M$ obtained as the inverse image of $\Rbb^+=(0, \infty)$ by the map $t$. Then we denote by $j$ the embedding $\GO\hookrightarrow\tX_M$, by $\widetilde p$ the map $p\circ j$, by $s$ the immersion $T_MX\hookrightarrow \tX_M$, by $\tau$ the projection $T_MX\to M$ and by $i$ the inclusion $M\hookrightarrow X$. We show the diagram
		\[
			\xymatrix{
				T_MX\ar[r]^{s}\ar[d]_{\tau}&\tX_M\ar[d]^{p}&\GO\ar[l]_{j}\ar[dl]^{\tp}\\
				M\ar[r]_{i}&X&
			}
		\]
		about these maps. For the embedding $j$, we remark that $j!$ is exact. 
		\begin{proposition}\cite[Lemma 4.2.1.]{KS90}
			There is a natural isomorphism
			\[
				s^{-1}Rj_*\tp^{-1}\simeq s^!j_!\tp^!. 
			\]
		\end{proposition}
		%\begin{proof}
			%%%%%%%%%%%%%
		%\end{proof}
		
		\begin{definition}
		\label{specialization}
			Let $F\in D^b(X)$. One sets
			\begin{equation}
				\nu_M(F):=s^{-1}Rj_*\tp^{-1}F\simeq s^!j_!\tp^!F, 
			\end{equation}
			and say that $\nu_M(F)$ is the specialization of $F$ along $M$. 
		\end{definition}
		
		Let $\tau:E\to M$ be a real vector bundle and $\pi:E^*\to M$ be the dual vector bundle. We denote by $p_1$ and $p_2$ the first and second projections from $E\times_{M}E^*$: 
		\[
			\xymatrix{
				&E\times_{M}E^*\ar[dl]_{p_1}\ar[dr]^{p_2}&\\
				E\ar[dr]_{\tau}&&E^*\ar[dl]^{\pi}\\
				&M.&
			}
		\]
		We define the subsets of $E\times_{M}E^*$ as below:
		\[
			P^+:=\{(p, q)\in E\times_ME^*; \langle p, q\rangle\geq0\}. %, \quad P^-:=\{(p, q)\in E\times_ME^*; \langle p, q\rangle\leq0\}.
		\]
		%Then we denote by $\tGY_{P^+}, \tGF_{P^+}, \tGY_{P^-}$ and $\tGF_{P^-}$ the following functors: 
		%\begin{eqnarray*}
		%	\tGY_{P^+}&=&Rp_{2*}\circ R\GG_{P^+}\circ p_1^{-1}\\
		%	\tGF_{P^+}&=&Rp_{1!}\circ (-)_{P^+}\circ p_2^!\\
		%	\tGY_{P^-}&=&Rp_{1*}\circ R\GG_{P^-}\circ p_2^!\\
		%	\tGF_{P^-}&=&Rp_{2!}\circ (-)_{P^-}\circ p_1^{-1}
		%\end{eqnarray*}
		
		%\begin{theorem}\cite[Theorem 3.7.7.]{KS90}
			
		%\end{theorem}
		\begin{definition}
			For $F\in D^b(E)$, 
			\[
				F^\wedge :=Rp_{2*}R\GG_{P^+}(p_1^{-1}F) 
			\]
			is called the Fourier-Sato transform of $F$. 
		\end{definition}
		
		For $F\in D^b(X)$, we recall that $\nu_M(F)$ is an object oj $D^b(T_MX)$. 
		
		\begin{definition}%\cite[Definition 4.3.1.]{KS90}
		\label{microlocalization}
			For $F\in D^b(X)$, the {\bf microlocalization} of $F$ along $M$ denoted by $\mu_M(F)$ is the Fourier-Sato transform of $\nu_M(F)$: 
			\begin{equation}
				\mu_M(F)=\nu_M(F)^\wedge. 
			\end{equation}
		\end{definition}
		
		%\mu hom
		
		\begin{definition}
		\label{def of muhom}
			Let $p_i:X\times X\to X$ denote the $i$-th projection and let $\GD\subset X\times X$ denote the diagonal of $X\times X$. For $F, G\in D^b(X)$, one sets
			\begin{equation}
				\mu hom(F, G):=\mu_{\GD}\RHOM(q_2^{-1}F, q_1^!G). 
			\end{equation}
		\end{definition}
		
	\subsection{Radon transform}
	\label{radom transform}\phantom{a}
		%cohomologically constructible
		%$BN(\GO_X, \GO_Y)$
		Let $X, Y, Z$ be manifolds and $\GO_X, \GO_Y, \GO_Z$ be open subsets of $T^*X, T^*Y, T^*Z$ respectively. 
		The {\bf antipodal morphism} on the cotangent bundle is $a:T^*X\to T^*X, a(x; \xi)\mapsto(x; -\xi)$. We set the antipodal subset $\GO_X^a:=\{(x; \xi)\in T^*X; (x; -\xi)\in\GO_X\}$. 
		We denote by $q_1:X\times Y\to X$ and $q_2: X\times Y\to Y$ the projections respectively, and by $p_1:T^*(X\times Y)\to T^*X$ and $p_2:T^*(X\times Y)\to T^*Y$ the projections respectively. We set the morphism $p_2^a:=a\circ p_2:T^*X\times T^*Y\to T^*Y$. 
		
		\begin{definition}
			Let $\BN(X, Y; \GO_X, \GO_Y)$ denote the full subcategory of $D^b(X\times Y; \GO_X\times T^*Y)$ consisting of objects $K$ satisfying: 
			\begin{enumerate}[(i)]
				\item $\SS(K)\cap(\GO_X\times T^*Y)\subset \GO_X\times\GO_Y^a$, 
				\item $p_1:\SS(K)\cap(\GO_X\times T^*Y)\to \GO_X$ is proper. 
			\end{enumerate}
			If there is no risk of confusion, we write $\BN(\GO_X, \GO_Y)$ instead of $\BN(X, Y; \GO_X, \GO_Y)$. 
			If $Y=\{\pt\}$, then $\BN(\GO, \{\pt\})\simeq D^b(X; \GO_X)$. 
		\end{definition}
		
		\begin{proposition}\cite[Proposition7.1.2.(iii)]{KS90}
			Assume $K\in\BN(\GO_X, \GO_Y)$ and $L\in\BN(\GO_Y, \GO_Z)$. Then $K\circ L\in\BN(\GO_X, \GO_Z)$. 
		\end{proposition}
		
		We denote by $r:X\times Y\to Y\times X$ the canonical morphism $r(x, y)=(y, x)$. 
		
		\begin{proposition}\cite[Proposition7.1.4.(i) and Proposition7.1.9.]{KS90}
		\label{replacement}
			Let $K\in D^b(X\times Y)$. 
			\begin{enumerate}[(i)]
				\item Assume $r_*K\in\BN(\GO_Y^a, \GO_X^a)$. The functor $\GY_K:D^b(X)\to D^b(Y)$ induces a well-defined functor from $D^b(X; \GO_X)$ to $D^b(Y; \GO_Y)$. 
				\item Assume $K$ is cohomologically constructible and $K\in\BN(\GO_X, \GO_Y)$. Set $K^*=r_*\RHOM(K, p_2^!\Bk_Y)$. Then $\GF_K\simeq \GY_{K^*}$ as functor from $D^b(Y; \GO_Y)$ to $D^b(X; \GO_X)$. 
			\end{enumerate}
		\end{proposition}
		
		\begin{proposition}\cite[Proposition7.1.11.]{KS90}
		\label{muhom_iso}
			Let $K\in D^b(X\times Y), F\in D^b(X; \GO_X), G\in D^b(Y; \GO_Y)$. 
			\begin{enumerate}[(i)]
				\item Assume $K\in\BN(\GO_X, \GO_Y)$. Then there is a natural isomorphism in $D^b(\GO_X)$: 
				\[
					Rp_{1*}\mu hom(K, \RHOM(q_2^{-1}G,  q_1^!F))\simeq \mu hom(\GF_K(G), F). 
				\]
				\item Assume $r_*K\in\BN(\GO_Y^a, \GO_X^a)$. Then there is a natural isomorphism in $D^b(\GO_Y)$: 
				\[
					Rp_{2*}^a\mu hom(K, \RHOM(q_2^{-1}G,  q_1^!F))\simeq \mu hom(G, \GY_K(F)). 
				\]
			\end{enumerate}
		\end{proposition}
		
		The subset $\GL\subset T^*X$ is called {\bf conic} if it is closed by the action $\Rbb^+$. 
		
		\begin{definition}
		\label{contact transform}
			Let $\GL\subset\GO_X\times\GO_Y^a$ be a closed conic subset, where $\GO_X$ and $\GO_Y$ are open subsets of $T^*X$ and $T^*Y$ respectively. Assume that
			\[
				p_1|_\GL:\GL\to \GO_X \text{\ and\ } p_2^a|_\GL:\GL\to \GO_Y
			\]
			are homeomorphisms. Then one sets $\chi:=p_1|_\GL\circ(p_2^a|_\GL)^{-1}:\GO_Y\to\GO_X$. It is called a {\bf contact transform} if $\GL$ is smooth and Lagrangian and $p_j$ is diffeomorphism. 
		\end{definition}

		\begin{definition}
		%\label{quantized contact transform}
			Let $X$ and $Y$ be two manifolds of the same dimension. Suppose $\GO_X\subset \dot T^*X$ and $\GO_Y\subset \dot T^*Y$ are two open subsets, and $\chi:\GO_Y\to\GO_X$ is a contact transform. A sheaf $K\in D^b(X\times Y)$ is a {\bf quantized contact transform} associated with $\chi$ from $(Y, \GO_Y)$ to $(X, \GO_X)$ if: 
			\begin{enumerate}[(i)]
				\item $K$ is cohomologically constructible, 
				\item $(p_1^{-1}(\GO_X)\cup (p^a_2)^{-1}(\GO_Y))\cap \SS(K)\subset \GL$, 
				\item the natural morphism $\Bk_\GL\to \mu hom(K, K)|_\GL$ is an isomorphism in $D^b(\GL)$. 
			\end{enumerate}
		\end{definition}
		
		\begin{theorem}\cite[Theorem 7.2.1.]{KS90}
		\label{quantized contact transform}
			Suppose $K\in D^b(X\times Y)$ is a quantized contact transform from $(Y, \GO_Y)$ to $(X, \GO_X)$. Then
			\begin{equation}
			 	\GF_K:D^b(Y; \GO_Y)\to D^b(X; \GO_X)
			\end{equation}
			is an equivalence of categories. 
		\end{theorem}
		
		We denote by $\GD_X\subset X\times X$ and $\GD_Y\subset Y\times Y$ the diagonal subsets respectively. 
		To prove the above theorem, we introduce the following lemma. 
		
		\begin{lemma}\cite[Proposition7.1.10.]{KS90}
		\label{adjoint_inverse}
			Let $K\in\BN(\GO_X, \GO_Y)$ and $L\in\BN(\GO_Y, \GO_X)$. Assume $K\circ L\simeq \Bk_{\GD_X}$ in $D^b(X\times X; \GO_X\times T^*X)$ and $L\circ K\simeq \Bk_{\GD_Y}$ in $D^b(Y\times Y; \GO_Y\times T^*Y)$. Then $\GF_K:D^b(Y; \GO_Y)\to D^b(X; \GO_X)$ and $\GF_L:D^b(X; \GO_X)\to D^b(Y; \GO_Y)$ are equivalences of categories, inverse to each other.  
		\end{lemma}
		
		\hspace{-1.5em}{\bf proof of Theorem 3.20.}
			
			Since $(p_1^{-1}(\GO_X)\cup (p^a_2)^{-1}(\GO_Y))\cap \SS(K)\subset \GL\subset \GO_X\times\GO_Y^a$, $\SS(K)\cap(\GO_X\times T^*Y)\subset\GO_X\times\GO_Y^a$. By the homeomorphism $p_1|_\GL:\GL\to \GO_X$, $\SS(K)\cap(\GO_X\times T^*Y)\to\GO_X$ is proper. Then $K\in\BN(\GO_X, \GO_Y)$ and $r_*K\in\BN(\GO_Y^a, \GO_X^a)$. 
			Consider the Cartesian square
			\[
			\xymatrix{
				X\times Y\ar[r]^-{\tilde{j}}\ar[d]_{q_2}&X\times Y\times Y\ar[d]^{q_{23}}\\
				Y\ar[r]_{j}\ar@{}@<0ex>[ur]|{\square}&Y\times Y, 
			}
			\]
			where $j$ and $\tilde j$ are the diagonal embeddings and $q_{ij}$ is the projection. Set $E=\RHOM(q_{12}^{-1}K, q_{13}^!K)$. This is an object of $D^b(X\times Y\times Y)$. Since $K$ is cohomologically constructible, setting $K^*=r_*\RHOM(K, q_1^!\Bk_X)$, we have $K^*\circ_{X} K\simeq Rq_{23*}E$ in $\BN(\GO_Y, \GO_Y)$ by Proposition \ref{replacement}. 
			On the other hand, we have $\tilde j^!E\simeq \RHOM(\tilde j^{-1}q_{12}^{-1}K, \tilde j^!q_{13}^!K)\simeq\RHOM(K, K)$. Hence we obtain the canonical morphism
			\[
				\Bk_{X\times Y}\to\RHOM(K, K)\to \tilde j^!E, 
			\]
			which induces
			\[
			\Bk_Y\to Rq_{2*}\Bk_{X\times Y}\to Rq_{2*}\tilde j^!E\simeq \tilde j^!Rq_{23*}E. 
			\]
			Therefore we get the following morphism
			\[
				\Ga:\Bk_{\GD_Y}\simeq Rj_!\Bk_Y\to K^*\circ K\quad \text{in}\quad \BN(\GO_Y, \GO_Y). 
			\]
			
			We shall prove that $\Ga$ is an isomorphism. Let $Z$ be another manifold, let $F\in D^b(X\times Z; \GO_X\times T^*Z)$, $G\in D^b(Y\times Z; \GO_Y\times T^*Z)$. We use $\chi:\GO_Y\times T^*Z\to \GO_X\times T^*Z$ to denote a contact transform which is induced by the contact transform $\GO_Y\to \GO_X$. For two projections $q_{13}^\prime:X\times Z\times Y\times Z\to X\times Y$ and $q_{24}^\prime:X\times Z\times Y\times Z\to Z\times Z$, we set 
			\[
				i_Z(K):=q_{13}^{\prime-1}K\otimes q_{24}^{\prime-1}\Bk_{\GD_Z}\quad \text{in}\quad \BN(\GO_X\times\GO_Z, \GO_Y\times\GO_Z). 
			\]
			By Proposition \ref{muhom_iso}(ii), we have a natural isomorphism in $D^b(\GO_X\times T^*Z)$: 
			\[
				\chi_*\mu hom(G, \GY_{i_Z(K)}(F))\simeq \mu hom(\GF_{i_Z(K)}(G), F). 
			\]
			Considering $Z=Y, G=\Bk_{\GD_Y}$ and $F=K$, we obtain the isomorphism
			\[
				\chi_*\mu hom(\Bk_{\GD_Y}, \GY_{i_Z(K)}(K))\simeq \mu hom(K, K). 
			\]
			Since $\GY_{i_Z(K)}(K)\simeq K^*\circ_X K$ by Proposition \ref{replacement}(ii), we obtain the isomorphism %on $\GO_Y\subset T^*_{\GD_Y}(Y\times Y)$: 
			\[
				\mu_{\GD_Y}(\Ga): \mu_{\GD_Y}(\Bk_{\GD_Y})\simeq \mu_{\GD_Y}(K^*\circ K). 
			\]
			Since $\SS(K^*\circ K)\cap(\GO_Y\times T^*Y)\subset T^*_{\GD_Y}(Y\times Y)$, Proposition \ref{microsupport_prop} and the above isomorphism imply that $\Ga$ is an isomorphism in $\BN(\GO_X, \GO_Y)$. We can prove similarly the isomorphism $\Bk_{\GD_X}\simeq K\circ K^{*\prime}$ in $\BN(\GO_X, \GO_X)$ where $K^{*\prime}=r_*\RHOM(K, q_2^!\Bk_Y)$. By Lemma \ref{adjoint_inverse}, we can see that $\GF_{K}:D^b(Y; \GO_Y)\to D^b(X; \GO_X)$ is an equivalence of categories. 
		\qed

		By \cite{Gao17}, the {\bf Radon transform} is an example of the above equivalence. We use $A$ to denote the subset $A: =\{(y, t, x)\in S^{n-1}\times \Rbb\times\Rbb^n;x\cdot y\leq t\}$. Let $\pi_1:S^{n-1}\times\Rbb\times\Rbb^n\to S^{n-1}\times\Rbb$ and $\pi_2:S^{n-1}\times\Rbb\times\Rbb^n\to\Rbb^n$ be projections. The Radon transform for sheaves is defined to be 
		\begin{equation}
			\GF_{\Bk_A}:D^b(\Rbb^n)\to D^b(S^{n-1}\times \Rbb), F\mapsto \Bk_A\circ F=R\pi_{1!}(\Bk_A\otimes \pi_2^{-1}F). 
		\end{equation}
		By Theorem \ref{quantized contact transform}, we get the following equivalence of categories. 
		\begin{theorem}\cite[Theorem 3.3.]{Gao17}
			The Radon transform for sheaves induces an equivalence between the localized sheaf categories
			\begin{equation}
				\GF_{\Bk_A}:D^b(\Rbb^n, \dot T^*\Rbb^n)\to D^b(S^{n-1}\times \Rbb, T^{*, +}(S^{n-1}\times\Rbb)). 
			\end{equation}
			Here $T^{*, +}(S^{n-1}\times\Rbb)=T^*S^{n-1}\times T^{*, +}\Rbb$ and $T^{*, +}\Rbb=\{(t; \tau)\in T^*\Rbb; \tau>0\}$. 
		\end{theorem}
		
		\begin{example}
			Let $\|\cdot\|$ denote the Euclidean norm on $\Rbb^n$. 
%			For $r\geq0$, we set some sheaves
%			\[
%				F=\Bk_{\{x\in\Rbb^n; \|x\|< r\}},\quad G=\Bk_{\{x\in\Rbb^n; \|x\|\leq r\}},\quad H=\Bk_{\{x\in\Rbb^n; \|x\|=r\}}
%			\]
			For $r\geq0$, we set $F=\Bk_{\{x\in\Rbb^n; \|x\|\leq r\}}$
			in $D^b(\Rbb^n)$. We calculate the compositions of $\Bk_A$ and
			% these sheaves respectively. 
			$F$. 
			We set $B:=A\cap\{(y, t, x)\in S^{n-1}\times \Rbb\times \Rbb^n; \|x\|\leq r\}$. 
			Since $\pi_1$ is proper on $B$, we have
			\begin{align*}
				%\Bk_A\circ G&\simeq R\pi_{1!}(\Bk_A\otimes\pi_2^{-1}G)\\
				\Bk_A\circ F&\simeq R\pi_{1!}(\Bk_A\otimes\pi_2^{-1}F)\\
				&\simeq R\pi_{1!}(\Bk_A\otimes\Bk_{S^{n-1}\times\Rbb\times\{x\in\Rbb^n; \|x\|\leq r\}})\\
				&\simeq R\pi_{1!}\Bk_{B}\\
				&\simeq R\pi_{1*}\Bk_B
			\end{align*}
			in $D^b(S^{n-1}\times \Rbb)$. Let $j:B\to S^{n-1}\times \Rbb\times \Rbb^n$ denote the inclusion. Since $j_*=j_!$, $j_*$ is exact by Proposition \ref{i!}. 
			%Then we obtain the morphism $\Bk_{S^{n-1}\times[-r, \infty)}\to \Bk_A\circ G$ 
			Then we obtain the morphism $\Bk_{S^{n-1}\times[-r, \infty)}\to \Bk_A\circ F$ 
			which is the image of identity in $\Hom_{D^b(B)}(\Bk_B,\Bk_B)$ by the following isomorphism 
			\begin{align*}
				%\Hom_{D^b(S^{n-1}\times \Rbb)}(\Bk_{S^{n-1}\times [-r, \infty)}, \Bk_A\circ G)&\simeq \Hom_{D^b(\Rbb^n)}(\pi_1^{-1}\Bk_{S^{n-1}\times [-r, \infty)}, \Bk_B)\\
				\Hom_{D^b(S^{n-1}\times \Rbb)}(\Bk_{S^{n-1}\times [-r, \infty)}, \Bk_A\circ F)&\simeq \Hom_{D^b(\Rbb^n)}(\pi_1^{-1}\Bk_{S^{n-1}\times [-r, \infty)}, \Bk_B)\\
				&\simeq \Hom_{D^b(\Rbb^n)}(\pi_1^{-1}\Bk_{S^{n-1}\times [-r, \infty)}, j_*j^{-1}\Bk_{S^{n-1}\times \Rbb\times\Rbb^n})\\
				&\simeq \Hom_{D^b(B)}(j^{-1}\pi_1^{-1}\Bk_{S^{n-1}\times [-r, \infty)}, j^{-1}\Bk_{S^{n-1}\times \Rbb\times\Rbb^n})\\
				&\simeq \Hom_{D^b(B)}(\Bk_{B}, \Bk_{B}). 
			\end{align*}
			For any $(y, t)\in S^{n-1}\times\Rbb$, since $B\cap\pi_1^{-1}(y, t)$ is contractible if $(y, t)\in S^{n-1}\times[-r, \infty)$ and $B\cap\pi_1^{-1}(y, t)=\emptyset$ if $(y, t)\not\in S^{n-1}\times[-r, \infty)$, we have
			\begin{align*}
				(\Bk_{S^{n-1}\times[-r, \infty)})_{(y, t)}&\simeq (R\pi_{1*}\Bk_B)_{(y, t)}\\
				&\simeq R\GG(\pi_1^{-1}(y, t); \Bk_B|_{\pi_1^{-1}(y, t)})\\
				%&\simeq (\Bk_A\circ G)_{(y, t)}. 
				&\simeq (\Bk_A\circ F)_{(y, t)}. 
			\end{align*}
			%By Proposition \ref{stalk_iso}, we have $\Bk_{S^{n-1}\times[-r, \infty)}\simeq \Bk_A\circ G$. 
			By Proposition \ref{stalk_iso}, we have $\Bk_{S^{n-1}\times[-r, \infty)}\simeq \Bk_A\circ F$. 
			
			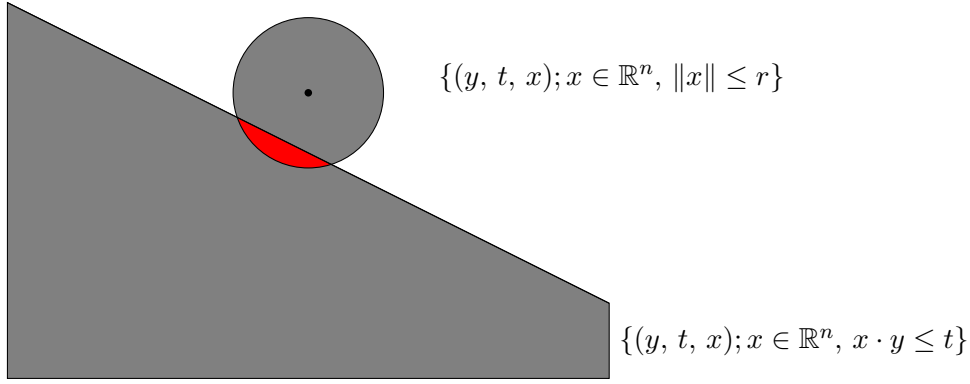
\begin{figure}[htbp]
				\centering
				\begin{tikzpicture}
				\coordinate[label=left:$0$](x)at(0, 1.8);
				\coordinate[label=right:$\{(y{\rm ,\ }t{\rm ,\ }x); x\in\Rbb^n{\rm ,\ } x\cdot y\leq t\}$](t)at(4, -1.5);
				\coordinate[label=right:$\{(y{\rm ,\ }t{\rm ,\ }x); x\in\Rbb^n{\rm ,\ } \|x\|\leq r\}$](ball)at(1.6, 2);
				\filldraw[fill=gray, opacity=0.1](-4, 3)--(-4,-2)--(4, -2)--(4, -1)--cycle;
				\fill[gray, opacity=0.1](x) circle (10mm);
				\fill[black, opacity=0.4](x) circle (0.5mm);
				\begin{scope}
					\clip (x) circle (10mm);
					\filldraw[fill=red, opacity=0.4](-4, 3)--(-4,-2)--(4, -2)--(4, -1)--cycle;
				\end{scope}
				\draw(x) circle (10mm);
				\draw (-4, 3)--(4, -1);
	 			\end{tikzpicture}			
				\caption{It is a figure on $\pi_1^{-1}(y, t)\simeq \Rbb^n$. If $(y, t)\in S^{n-1}\times[-r, \infty)$, $B\cap\pi_1^{-1}(y, t)$ is contractible. }
			\end{figure}
			
%			Similarly, we get $\Bk_{S^{n-1}\times[-r, \infty)}\oplus\Bk_{S^{n-1}\times[r, \infty)}[1-n]\to\Bk_A\circ H$, where $[l]:D^b(S^{n-1}\times\Rbb)\to D^b(S^{n-1}\times\Rbb)$ is a shift functor for $l\in\Zbb$. By Proposition \ref{distinguish triangle}, there is a distinguish triangle
%			\[
%				F\to G\to H\overset{+1}{\to}
%			\]
%			in $D^b(\Rbb^n)$. Applying the Radon transform, we obtain
%			\[
%				\Bk_A\circ F\to \Bk_{S^{n-1}\times[-r, \infty)}\to \Bk_{S^{n-1}\times[-r, \infty)}\oplus\Bk_{S^{n-1}\times[r, \infty)}[1-n]\overset{+1}{\to}. 
%			\]
%			Then we have $\Bk_A\circ F\simeq \Bk_{S^{n-1}\times[r, \infty)}[-n]$. 
			
		\end{example}
	
\section{Convolution and interleaving distance}
\label{CD}\phantom{a}
	Throughout this chapter, $\Vbb$ is a finite dimensional real vector space and is given a Euclidean structure. Let $\|\cdot\|$ denote the norm on $\Vbb$. For $r\geq0$, one sets
	\begin{equation}
		B_r=\{x\in\Vbb; \|x\|\leq r\},\quad \inte B_r=\{x\in\Vbb; \|x\|<r\}. 
	\end{equation}
	For $a\in\Rbb$, one defines
	\begin{equation}
		K_a=
		\left\{
		\begin{array}{cc}
			\Bk_{B_a}&{\rm if}\ a\geq0, \\
			\Bk_{\inte B_{-a}}[\dim \Vbb]&{\rm if}\ a<0
		\end{array}
		\right. 
	\end{equation}
	in $D^b(\Vbb)$. Here $[l]$ is the shift functor of degree $l\in \Zbb$. 

	\subsection{Comparison of convolution and kernel composition}
	\label{comparison}\phantom{a}
		%Convolution
		We define the convolution on $D^b(\Vbb)$. 
		\begin{definition}
		\label{convolution}
			Let $q_1, q_2, s:\Vbb\times\Vbb\to\Vbb$ as follows: 
			\begin{equation}
				q_1(v_1, v_2)=v_1, \ q_2(v_1, v_2)=v_2, \ s(v_1, v_2)=v_1+v_2. 
			\end{equation}
			One defines the convolution bifunctor by setting $F*G:=Rs_!(q_1^{-1}F\otimes q_2^{-1}G)$, where $F, G\in D^b(\Vbb)$. 	
		\end{definition}
		The following proposition is proved in \cite[\S 3.1]{KS18}. 
		\begin{proposition}
		\label{property of Ka}
			Let $a, b\in\Rbb$ and $F\in D^b(\Vbb)$. There are functorial isomorphisms
			\begin{equation}
				K_a*(K_b*F)\simeq K_{a+b}*F,\quad K_0*F\simeq F. 
			\end{equation}
		\end{proposition}
		For $a\geq b$ and $F\in D^b(\Vbb)$, there are canonical morphisms
		\begin{align*}
		\label{canonical morphism of K_a}
			\chi_{a, b}^K&:K_a\to K_b, \\
			\chi_{a, b}^K*F&:K_a*F\to K_b*F. 
		\end{align*}
		In particular, one has $\chi_{a, 0}^K*F:K_a*F\to F$ when $a\geq0$. 
		\begin{definition} 
		\label{a-isomorphism}
			\begin{enumerate}[(i)]
				\item Let $F, G\in D^b(\Vbb)$ and let $a\geq0$. The sheaves $F, G$ are {\bf $a$-interleaved} if there are two morphisms $f:K_a*F\to G$ and $g:K_a*G\to F$ such that the diagrams
				\begin{equation}
				\label{a-iso-dia}
					\xymatrix{
						K_{2a}*F\ar[rr]^{\chi_{2a,0}^K*F}\ar[dr]_{K_{a}*f}&&F\\
						&K_a*G, \ar[ur]_{g}\ar@{}[u]|{\circlearrowright}&
					}\quad 
					\xymatrix{
						&K_a*F\ar[dr]^{f}\ar@{}[d]|{\circlearrowright}&\\
						K_{2a}*G\ar[rr]_{\chi_{2a,0}^K*G}\ar[ur]^{K_{a}*g}&&G
						}
				\end{equation}
				are commutative. The pair of morphisms $(f, g)$ is called an {\bf $a$-interleaving}. 
				\item One sets 
				\begin{equation}
				%\label{convolution distance}
					 d_C(F, G):=\inf(\{a\geq0; F\ {\rm and}\ G\ {\rm are}\ a{\text -}{\rm interleaved}\}\cup\{\infty\}). 
				\end{equation}
				It is called the {\bf convolution distance} on $D^b(\Vbb)$. 
			\end{enumerate}
		\end{definition}
		Indeed, the convolution distance is a pseudo-extended distance. Namely, for $F, G, H\in D^b(\Vbb)$, $d_C:D^b(\Vbb)\times D^b(\Vbb)\to \Rbb_{\geq0}\cup\{\infty\}$ is a non-negative extended real-value function which satisfies
		\begin{enumerate}[(i)]
			\item $d_C(F, F)=0$, 
			\item $d_C(F, G)=d_C(G, F)$, 
			\item $d_C(F, H)\leq d_C(F, G)+d_C(G, H)$. 
		\end{enumerate}
		Berkouk and Ginot \cite[Proposition 6.9]{BG22} showed that there exists a pair of sheaves $F, G$ on $\Rbb$ such that $d_C(F, G)=0$ and $F\not\simeq G$. 
		
		The convolution functor $K_a*(-)$ can be represented  by a composition functor. 
		\begin{proposition}\cite[Proposition 2.1.10]{PS23}
		\label{conv and comp on V}
			For $a\geq0$, one sets $\GD_a=\{(v_1, v_2)\in\Vbb\times\Vbb;\|v_1-v_2\|\leq a\}$. Then there is a canonical  isomorphism of functor
			\begin{equation}
				\Bk_{\GD_a}\circ (-)\simeq K_a*(-). 
			\end{equation}
		\end{proposition}
		%We will give a proof of Proposition \ref{conv and comp on MV} which is a generalized proposition of the above that.  
			%\begin{proof}
			%	Let 
			%\end{proof}
		\begin{remark}
		\label{rewrite dia1}
		We can write the following diagrams instead of \eqnref{a-iso-dia}: 
			\begin{equation}
			\label{cmp-a-iso}
				\xymatrix{
					\Bk_{\GD_{2a}}\circ F\ar[rr]%^{\chi_{2a,0}*F}
					\ar[dr]_{\Bk_{\GD_{a}}\circ f}&&F\\
					&\Bk_{\GD_{a}}\circ G, \ar[ur]_{g}\ar@{}[u]|{\circlearrowright}&
				}\quad 
				\xymatrix{
					&\Bk_{\GD_{a}}\circ F\ar[dr]^{f}\ar@{}[d]|{\circlearrowright}&\\
					\Bk_{\GD_{2a}}\circ G\ar[rr]%^{\chi_{2a,0}*G}
					\ar[ur]^{\Bk_{\GD_{a}}\circ g}&&G. \\
				}
			\end{equation}
		Here the morphisms $\Bk_{\GD_{2a}}\circ F\to F$ and $\Bk_{\GD_{2a}}\circ G\to G$ are induced by canonical morphism $\chi_{2a, 0}^\GD:\Bk_{\GD_{2a}}\to\Bk_{\GD_0}$. 
		\end{remark}	
			
		%tilde convolution	
		Let $M$ be a smooth manifold. In \cite{GS14},  the convolution on $M\times\Vbb$ is similarly defined. 
		\begin{definition}
		\label{tilde convolution}
			Let $\widetilde{q}_1, \widetilde{q}_2, \widetilde{s}:M\times\Vbb\times\Vbb\to M\times\Vbb$ be as follows: 
			\begin{equation}
				\widetilde q_1(x, v_1, v_2)=(x, v_1),\ \widetilde q_2(x, v_1, v_2)=(x, v_2),\ \widetilde s(x, v_1, v_2)=(x, v_1+v_2). 
			\end{equation}
			One defines the tilde convolution bifunctor by setting $F\tcv G:=R\widetilde s_!(\widetilde q_1^{\ -1}F\otimes \widetilde q_2^{\ -1}G)$, where $F, G\in D^b(M\times \Vbb)$. 	
		\end{definition}
		%The following proposition is proved in (\cite{KS18}, \S 3.1). 

		For $a\in\Rbb$, one defines
		\begin{equation}
			L_a=
			\left\{
			\begin{array}{cc}
				\Bk_{M\times B_a}&{\rm if}\ a\geq0, \\
				\Bk_{M\times \inte B_{-a}}[\dim \Vbb]&{\rm if}\ a<0
			\end{array}
			\right. 
		\end{equation}
		in $D^b(M\times\Vbb)$. 

		\begin{proposition}
		\label{property of La}
			Let $a, b\in\Rbb$ and $F\in D^b(M\times \Vbb)$. There are functorial isomorphisms
			\begin{equation}
				L_a\tcv (L_b\tcv F)\simeq L_{a+b}\tcv F,\quad L_0\tcv F\simeq F. 
			\end{equation}
		\end{proposition}

		For $a\geq b$ and $F\in D^b(M\times \Vbb)$, there are canonical morphisms
		\begin{align*}
		%\label{canonical morphism of K_a}
			\chi_{a, b}^L&:L_a\to L_b, \\
			\chi_{a, b}^L\tcv F&:L_a\tcv F\to L_b\tcv F. 
		\end{align*}
		In particular, one has $\chi_{a, 0}^L\tcv F:L_a\tcv F\to F$ when $a\geq0$. 
		\begin{definition} 
		\label{a-interleaved}
			\begin{enumerate}[(i)]
				\item Let $F, G\in D^b(M\times \Vbb)$ and let $a\geq0$. The sheaves $F, G$ are {\bf $a$-interleaved} if there are two morphisms $f:L_a\tcv F\to G$ and $g:L_a\tcv G\to F$ such that the diagrams
				\begin{equation}
				\label{a-int-dia}
					\xymatrix{
						L_{2a}\tcv F\ar[rr]^{\chi_{2a,0}^L\tcv F}\ar[dr]_{L_{a}\tcv f}&&F\\
						&L_a\tcv G, \ar[ur]_{g}\ar@{}[u]|{\circlearrowright}&
					}\quad 
					\xymatrix{
						&L_a\tcv F\ar[dr]^{f}\ar@{}[d]|{\circlearrowright}&\\
						L_{2a}\tcv G\ar[rr]_{\chi_{2a,0}^L\tcv G}\ar[ur]^{L_{a}\tcv g}&&G\\
					}
				\end{equation}
				are commutative. The pair of morphisms $(f, g)$ is called an {\bf $a$-interleaving}. 
				\item One sets 
				\begin{equation}
				%\label{interleaving distance}
					 d_I(F, G):=\inf(\{a\geq0; F\ {\rm and}\ G\ {\rm are}\ a{\text -}{\rm interleaved}\}\cup\{\infty\}). 
				\end{equation}
				It is called the {\bf interleaving distance} on $D^b(M\times\Vbb)$. 
			\end{enumerate}
		\end{definition}
		%Indeed, the convolution distance is a pseudo-extended distance. 
		
		The tilde convolution functor $L_a\tcv (-)$ can be represented  by a tilde composition functor. 
		\begin{proposition}
		\label{conv and comp on MV}
			For $a\geq0$, one sets $Z_a=\{(x, v_1, v_2)\in M\times \Vbb\times\Vbb;\|v_1-v_2\|\leq a\}$. Then there is a canonical  isomorphism of functors
			\begin{equation}
				\Bk_{Z_a}\tcp (-)\simeq L_a\tcv (-). 
			\end{equation}
		\end{proposition}
		\begin{proof}
			Consider a morphism
			\begin{equation}
				\widetilde u:M\times\Vbb\times\Vbb\to M\times\Vbb\times\Vbb, (x, v_1, v_2)\mapsto (x, v_1-v_2, v_2). 
			\end{equation}
			This satisfies that 
			\begin{equation}
				\widetilde s\circ \widetilde u=\widetilde q_{1},\ \widetilde q_2\circ \widetilde u=\widetilde q_2, 
			\end{equation}
			\begin{equation}
				\widetilde u^{\ -1}\widetilde q_1^{\ -1}(M\times B_a)=\widetilde u^{\ -1}(M\times B_a\times \Vbb)=Z_a. 
			\end{equation}
			Moreover, there is a canonical isomorphism of functors $\widetilde u_!\widetilde u^{-1}\simeq\id$ since $\widetilde u$ is a homeomorphism. For $F\in D^b(M\times\Vbb)$, we have
			\begin{align*}
				L_a\tcv F&= R\widetilde s_!(\widetilde q_1^{\ -1}L_a\otimes \widetilde q_2^{\ -1}F)\\
				&\simeq R\widetilde s_!R\widetilde u_!\widetilde u^{\ -1}(\widetilde q_1^{\ -1}L_a\otimes \widetilde q_2^{\ -1}F)\\
				&\simeq R\widetilde s_!R\widetilde u_!(\widetilde u^{\ -1}\widetilde q_1^{\ -1}L_a\otimes \widetilde u^{\ -1}\widetilde q_2^{\ -1}F)\\
				&\simeq R\widetilde q_{1!}(\Bk_{Z_a}\otimes \widetilde q_2^{\ -1}F)\\
				&= \Bk_{Z_a}\tcp F. 
			\end{align*} 
		\end{proof}
		We can obtain the proof of Proposition \ref{conv and comp on V} when $M=\pt$. 
		\begin{remark}
		We can write the following diagrams instead of \eqnref{a-int-dia} as well as Remark \ref{rewrite dia1}:
			\begin{equation}
			\label{tcp-a-int}
				\xymatrix{
					\Bk_{Z_{2a}}\tcp F\ar[rr]%^{\chi_{2a,0}*F}
					\ar[dr]_{\Bk_{Z_{a}}\tcp f}&&F\\
					&\Bk_{Z_{a}}\tcp G, \ar[ur]_{g}\ar@{}[u]|{\circlearrowright}&
				}\quad 
				\xymatrix{
					&\Bk_{Z_{a}}\tcp F\ar[dr]^{f}\ar@{}[d]|{\circlearrowright}&\\
					\Bk_{Z_{2a}}\tcp G\ar[rr]%^{\chi_{2a,0}*G}
					\ar[ur]^{\Bk_{Z_{a}}\tcp g}&&G. \\
				}
			\end{equation}
		Here the morphisms $\Bk_{Z_{2a}}\tcp F\to F$ and $\Bk_{Z_{2a}}\tcp G\to G$ are induced by canonical morphism $\chi_{2a, 0}^Z:\Bk_{Z_{2a}}\to\Bk_{Z_0}$. 
		\end{remark}
		
	\subsection{Distance on localized category}
	\label{dlc}\phantom{a}
		The tilde convolution bifunctor is extended to the localized category $D^b(M\times\Vbb, T^{*, \Gg}(M\times\Vbb))$. Therefore, the interleaving distance $d_I$ induces a pseudo-extended distance on $D^b(M\times\Vbb, T^{*, \Gg}(M\times\Vbb))$. In particular, the convolution distance $d_C$ induces a pseudo-extended distance on the localized category $D^b(\Vbb, \dot T^*\Vbb)$. 
		\begin{proposition}\cite[Corollary 4.14.]{GS14}
		\label{microsupport and convolution}
			Let $F, G\in D^b(M\times\Vbb)$. Assume that there exist closed cones $A, B\subset\Vbb^*$ such that $\SS(F)\subset T^*M\times \Vbb\times A$ and $\SS(G)\subset T^*M\times \Vbb\times B$. Then 
			\begin{equation}
				\SS(F\tcv G)\subset T^*M\times \Vbb\times (A\cap B). 
			\end{equation}
		\end{proposition}
		This proposition implies that the convolution$(-)\tcv(-):D^b(M\times\Vbb)\times D^b(M\times\Vbb)\to D^b(M\times\Vbb)$ induces 
		\begin{equation}
			(-)\tcv(-):D^b(M\times\Vbb, T^{*, \Gg}(M\times\Vbb))\times D^b(M\times\Vbb, T^{*, \Gg}(M\times\Vbb))\to D^b(M\times\Vbb, T^{*, \Gg}(M\times\Vbb)). 
		\end{equation}
		When $M=\pt$ and $\Gg=\Vbb$, especially, we can obtain
		\begin{equation}
			(-)*(-):D^b(\Vbb, \dot T^*\Vbb)\times D^b(\Vbb, \dot T^*\Vbb)\to D^b(\Vbb, \dot T^*\Vbb). 
		\end{equation}
		\begin{definition}
		\label{dLC and dLI}
			\begin{enumerate}[(i)] 
				\item Let $F, G\in D^b(\Vbb, \dot T^*\Vbb)$ and let $a\geq0$. The sheaves $F, G$ are {\bf locally $a$-interleaved} if there are two morphisms $f:K_a*F\to G$ and $g:K_a*G\to F$ such that the diagrams \eqnref{a-iso-dia} are commutative. The pair of morphisms $(f, g)$ is called a {\bf local $a$-interleaving}. 
				\item One sets 
				\begin{equation}
				%\label{convolution distance}
					 d_{LC}(F, G):=\inf(\{a\geq0; F\ {\rm and}\ G\ {\rm are\ locally}\ a{\text -}{\rm interleaved}\}\cup\{\infty\}). 
				\end{equation}
				It is called the {\bf localized convolution distance} on $D^b(\Vbb, \dot T^*\Vbb)$. 
				\item Let $F, G\in D^b(M\times\Vbb, T^{*, \Gg}(M\times\Vbb))$ and let $a\geq0$. The sheaves $F, G$ are {\bf locally $a$-interleaved} if thre are two morphisms $f:L_a\tcv F\to G$ and $g:L_a\tcv G\to F$ such that the diagrams \eqnref{a-int-dia} are commutative. The pair of morphisms $(f, g)$ is called a {\bf local $a$-interleaving}. 
				\item One sets 
				\begin{equation}
				%\label{interleaving distance}
					 d_{LI}(F, G):=\inf(\{a\geq0; F\ {\rm and}\ G\ {\rm are\ locally}\ a{\text -}{\rm interleaved}\}\cup\{\infty\}). 
				\end{equation}
				It is called the {\bf localized interleaving distance} on $D^b(M\times\Vbb, T^{*, \Gg}(M\times\Vbb))$. 
			\end{enumerate}
		\end{definition}
		\begin{proposition}
		\label{inequality of distance}
			\begin{enumerate}[(i)]
				\item Let $F, G\in D^b(\Vbb)$. Then
				\begin{equation}
					d_{LC}(F, G)\leq d_C(F, G). 
				\end{equation}
				\item Let $F, G\in D^b(M\times \Vbb)$. Then
				\begin{equation}
					d_{LI}(F, G)\leq d_I(F, G). 
				\end{equation}
			\end{enumerate}
		\end{proposition}
		\begin{proof}
		(i) Assume that $F$ and $G$ are $a$-interleaved. Applying the localization functor, we get the diagram \eqnref{a-iso-dia} in $D^b(\Vbb, \dot T^*\Vbb)$. 
		This proves that $d_{LC}(F, G)\leq d_C(F, G)$. \\
		(ii) It can be proved in the similar way. 
		\end{proof}
		
\section{Isometry theorem}
\label{IT}
\phantom{a}
	This section presents the proof of the isometry theorem by the Radon transform. We can define the localized convolution (resp.\,interleaving) distance on $D^b(\Rbb^n, \dot{T}^*\Rbb^n)$ (resp.\,$D^b(S^{n-1}\times \Rbb, T^{*, +}(S^{n-1}\times\Rbb))$) by using Definition \ref{dLC and dLI}. At the beginning, the convolution (resp.\,interleaving) distance is defined by convolution (resp.\,tilde convolution) functor $K_a*(-)$ (resp.\,$L_a\tcv (-)$). These functors can be represented  by composition functors: $\Bk_{\GD_a}\circ(-)$ and $\Bk_{Z_a}\tcp(-)$. In order to prove the isometry theorem, we prove the commutativity between these functors and the Radon transform. 
	%$\Vbb=\Rbb^n, M\times\Vbb=S^{n-1}\times\Rbb$
	\subsection{Radon isometry theorem}
	\label{radon isometry}%\phantom{a}
	%compatibility between Radon transform and kernel
		\begin{lemma}
		\label{commutative between radon and cmp}
			For any $a\geq0$, there are canonical isomorphisms
			\begin{equation}
				\Bk_A\circ \Bk_{\GD_a}\simeq \Bk_{A_a}\simeq \Bk_{Z_a}\tcp \Bk_A\quad {\rm in}\quad D^b(S^{n-1}\times\Rbb\times\Rbb^n), 
			\end{equation} 
			where $A_a=\{(y, t, x)\in S^{n-1}\times\Rbb\times\Rbb^n; x\cdot y\leq t+a\}$. 
		\end{lemma}
		\begin{proof}
			First, we denote by $p_{ij}$ the following projection: 
			\begin{equation}
				\xymatrix{
					&(S^{n-1}\times\Rbb)\times\Rbb^n_{x_2}\times\Rbb^n_{x_3}\ar[dl]_{p_{12}}\ar[d]_{p_{13}}\ar[dr]^{p_{23}}&\\
					(S^{n-1}\times\Rbb)\times\Rbb^n_{x_2}&(S^{n-1}\times\Rbb)\times\Rbb^n_{x_3}&\Rbb^n_{x_2}\times\Rbb^n_{x_3}. 
				}
			\end{equation}
			We set $A^\prime_a:=\{(y, t, x_2, x_3)\in S^{n-1}\times\Rbb\times\Rbb^n\times\Rbb^n; \|x_2-x_3\|\leq a, x_2\cdot y\leq t\}$. Then, we get $p_{12}^{-1}\Bk_A\otimes p_{23}^{-1}\Bk_{\GD_a}\simeq \Bk_{A_a^\prime}$ and $p_{13}^{-1}(A_a)\supset A_a^\prime$. Let $i: A^\prime_a\to S^{n-1}\times\Rbb\times\Rbb^n\times\Rbb^n$ be an inclusion map. Then $p_{13}|_{A^\prime_a}:A^\prime_a\to S^{n-1}\times\Rbb\times\Rbb^n$ is proper. Therefore we obtain the following morphism
			\begin{align*}
				\Bk_A\circ \Bk_{\GD_a}&=Rp_{13!}(p_{12}^{-1}\Bk_A\otimes p_{23}^{-1}\Bk_{\GD_a})\\
				&\simeq Rp_{13!}\Bk_{A^\prime_a}\\
				&\simeq Rp_{13*}\Bk_{A^\prime_a}\\
				&\simeq Rp_{13*}(\Bk_{S^{n-1}\times\Rbb\times\Rbb^n\times\Rbb^n})_{A^\prime_a}\\
				&\simeq Rp_{13*}((\Bk_{S^{n-1}\times\Rbb\times\Rbb^n\times\Rbb^n})_{p_{13}^{-1}(A_a)})_{A^\prime_a}\\
				&\simeq Rp_{13*}Ri_*i^{-1}p_{13}^{-1}\Bk_{A_a}\\
				&\simeq R(p_{13}\circ i)_*(p_{13}\circ i)^{-1}\Bk_{A_a}\\
				&\leftarrow \Bk_{A_a}. 
			\end{align*}
			Moreover $A^\prime_a\cap {p_{13}}^{-1}(y, t, x)=A^\prime_a\cap \{(y, t, x_2, x); x_2\in\Rbb^n\}$ is contractible if $(y, t, x)\in A_a$ and $A^\prime_a\cap {p_{13}}^{-1}(y, t, x)=\emptyset$ if $(y, t, x)\notin A_a$. We can prove that it induces the following isomorphism on stalks at any $(y, t, x)\in S^{n-1}\times\Rbb\times\Rbb^n$: 
			\begin{align*}
				(\Bk_A\circ \Bk_{\GD_a})_{(y, t, x)}&\simeq (Rp_{13*}\Bk_{A^\prime_a})_{(y, t, x)}\\
				&\simeq R\GG(\{(y, t, x_2, x); x_2\in\Rbb^n\};\Bk_{A^\prime_a}|_{\{(y, t, x_2, x); x_2\in\Rbb^n\}})\\
				&\simeq (\Bk_{A_a})_{(y, t, x)}. 
			\end{align*}
		
			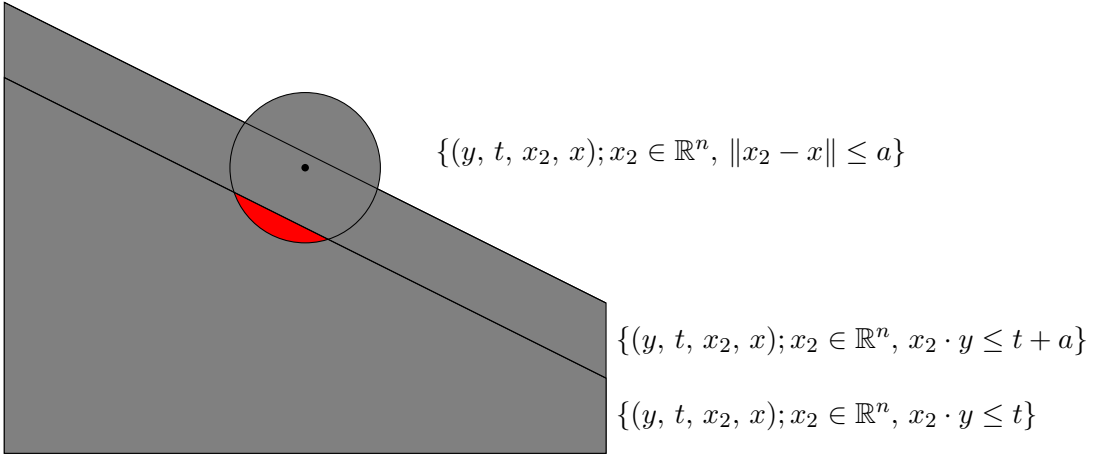
\begin{figure}[htbp]
				\centering
				\begin{tikzpicture}
				\coordinate[label=left:$x$](x)at(0, 1.8);
				\coordinate[label=right:$\{(y{\rm ,\ }t{\rm ,\ }x_2{\rm ,\ }x); x_2\in\Rbb^n{\rm ,\ } x_2\cdot y\leq t\}$](t)at(4, -1.5);
				\coordinate[label=right:$\{(y{\rm ,\ }t{\rm ,\ }x_2{\rm ,\ }x); x_2\in\Rbb^n{\rm ,\ } x_2\cdot y\leq t+a\}$](t+a)at(4, -0.5);
				\coordinate[label=right:$\{(y{\rm ,\ }t{\rm ,\ }x_2{\rm ,\ }x); x_2\in\Rbb^n{\rm ,\ } \|x_2-x\|\leq a\}$](ball)at(1.6, 2);
				\filldraw[fill=gray, opacity=0.1](-4, 4)--(-4,-2)--(4, -2)--(4, 0)--cycle;
				\filldraw[fill=gray, opacity=0.1](-4, 3)--(-4,-2)--(4, -2)--(4, -1)--cycle;
				\fill[gray, opacity=0.1](x) circle (10mm);
				\fill[black, opacity=0.4](x) circle (0.5mm);
				\begin{scope}
					\clip (x) circle (10mm);
					\filldraw[fill=red, opacity=0.4](-4, 3)--(-4,-2)--(4, -2)--(4, -1)--cycle;
				\end{scope}
				\draw(x) circle (10mm);
				\draw (-4, 3)--(4, -1);
				\draw (-4, 4)--(4, 0);
	 			\end{tikzpicture}			
				\caption{It is a figure on $p_{13}^{-1}(y, t, x)\simeq \Rbb^n_{x_2}$. If $(y, t, x)\in A_a$, $A^\prime_a\cap {p_{13}}^{-1}(y, t, x)$ is contractible. }
			\end{figure}
		
			We get the second isomorphism in the same way. 
		
			\hspace{-1.5em}{\bf Another proof. }
			We have 
			\begin{align*}
				\Hom_{D^b(S^{n-1}\times\Rbb\times\Rbb^n)}(\Bk_{A_a}, \Bk_A\circ \Bk_{\GD_a})
				&\simeq \Hom_{D^b(S^{n-1}\times\Rbb\times\Rbb^n)}(\Bk_{A_a}, Rp_{13!}(p_{12}^{-1}\Bk_A\otimes p_{23}^{-1}\Bk_{\GD_a}))\\
				&\simeq \Hom_{D^b(S^{n-1}\times\Rbb\times\Rbb^n\times\Rbb^n)}(p_{13}^{-1}\Bk_{A_a}, p_{12}^{-1}\Bk_A\otimes p_{23}^{-1}\Bk_{\GD_a})\\
				&\simeq \Hom_{D^b(S^{n-1}\times\Rbb\times\Rbb^n\times\Rbb^n)}(p_{13}^{-1}\Bk_{A_a}, \Bk_{A_a^\prime})\\
				&\simeq \Hom_{D^b(S^{n-1}\times\Rbb\times\Rbb^n\times\Rbb^n)}(p_{13}^{-1}\Bk_{A_a}, Ri_*i^{-1}\Bk_{S^{n-1}\times\Rbb\times\Rbb^n\times\Rbb^n})\\
				&\simeq \Hom_{D^b(A_a^\prime)}(i^{-1}p_{13}^{-1}\Bk_{A_a}, i^{-1}\Bk_{S^{n-1}\times\Rbb\times\Rbb^n\times\Rbb^n})\\
				&\simeq \Hom_{D^b(A_a^\prime)}(\Bk_{A_a^\prime}, \Bk_{A_a^\prime}). \\
			\end{align*}
			We can obtain the morphism $\Bk_{A_a}\to \Bk_A\circ\Bk_{\GD_a}$ considering the image of $\id_{\Bk_{A^\prime}}$ in $\RHom_{D^b(S^{n-1}\times\Rbb\times\Rbb^n)}(\Bk_{A_a}, \Bk_A\circ \Bk_{\GD_a})$. The rest of this proof can be done in the same way as the previous one. 
		\end{proof}
		\begin{lemma}
		\label{commutative between radon and chi}
			For $a\geq b\geq0$, there is a canonical isomorphism
			\begin{equation}
				\Bk_A\circ\chi^\GD_{a, b}\simeq \chi^Z_{a, b}\tcp\Bk_A. 
			\end{equation}
			%notationどうすんねん
		\end{lemma}
		\begin{proof}
			By constructions of canonical morphisms $\chi^\GD_{a, b}$ and $\chi^Z_{a, b}$, we can prove the statement immediately. 
		\end{proof}
		
		%d_I\leq d_C
		
		The following theorem is the same as Theorem \ref{main theorem}. 
		\begin{theorem}
		\label{re:main theorem}
			The Radon transform induces an isometric functor between $d_{LC}$ and $d_{LI}$:
			\begin{equation}
				\GF_{\Bk_A}:(D^b(\Rbb^n, \dot{T}^*\Rbb^n), d_{LC})\to (D^b(S^{n-1}\times\Rbb, T^{*, +}(S^{n-1}\times\Rbb)), d_{LI}). 
			\end{equation}
			Namely, for any $F, G\in D^b(\Rbb^n, \dot{T}^*\Rbb^n)$, $d_{LC}(F, G)=d_{LI}(\GF_{\Bk_A}(F), \GF_{\Bk_A}(G))$. 
		\end{theorem}
		\begin{proof}
			Assume that $F$ and $G$ are locally $a$-interleaved. We show that $\GF_{\Bk_A}(F)$ and $\GF_{\Bk_A}(G)$ are locally $a$-interleaved. Applying the Radon transform $\GF_{\Bk_A}=\Bk_A\circ(-)$ to the diagram \eqnref{cmp-a-iso},  we have the following diagrams
			\begin{equation*}
			%\label{cmp-a-iso}
				\xymatrix{
					\Bk_A\circ (\Bk_{\GD_{2a}}\circ F)\ar[rr]%^{\chi_{2a,0}*F}
					\ar[dr]_{\Bk_A\circ (\Bk_{\GD_{a}}\circ f)\quad}&&\Bk_A\circ F\\
					&\Bk_A\circ (\Bk_{\GD_{a}}\circ G), \ar[ur]_{\Bk_A\circ g}\ar@{}[u]|{\circlearrowright}&	
				}
			\end{equation*}
			\begin{equation*}
				\xymatrix{
					&\Bk_A\circ (\Bk_{\GD_{a}}\circ F)\ar[dr]^{\Bk_A\circ f}\ar@{}[d]|{\circlearrowright}&\\
					\Bk_A\circ (\Bk_{\GD_{2a}}\circ G)\ar[rr]%^{\chi_{2a,0}*G
					\ar[ur]^{\Bk_A\circ (\Bk_{\GD_{a}}\circ g)\quad}&&\Bk_A\circ G. 
				}
			\end{equation*}
			By Lemma \ref{commutative between radon and cmp}, Proposition \ref{associativity of composition} and Proposition \ref{associativity of tilde composition}, we have
			\[
				\Bk_A\circ(\Bk_{\GD_a}\circ F)\simeq (\Bk_A\circ \Bk_{\GD_a})\circ F\simeq (\Bk_{Z_a}\tcp \Bk_A)\circ F\simeq \Bk_{Z_a}\tcp(\Bk_A\circ F). 
			\]
			By Lemma \ref{commutative between radon and chi}, we obtain locally $a$-interleaved between $\GF_{\Bk_A}(F)$ and $\GF_{\Bk_A}(G)$. Therefore an inequality $d_{LI}(\GF_{\Bk_A}(F), \GF_{\Bk_A}(G))\leq d_{LC}(F, G)$ holds. 
		
			Conversely, since $\GF_{\Bk_A}$ is an equivalence of categories, $F$ and $G$ are locally $a$-interleaved if $\GF_{\Bk_A}(F)$ and $\GF_{\Bk_A}(G)$ are locally $a$-interleaved. Therefore, $d_{LI}(\GF_{\Bk_A}(F), \GF_{\Bk_A}(G))\geq d_{LC}(F, G)$ holds. 
		\end{proof}
		
	\subsection{Example and conjecture about distances}
	\label{conjecture}\phantom{a}
		This section deals with the closedness of distances $d_C$ and $d_I$ for constructible sheaves (see \cite[Definition 8.4.3.]{KS90}). For a manifold $X$, we revise $D^b(X)$ as the bounded derived category of sheaves of modules whose coefficients are rings. 
		\begin{conjecture}
		\label{conj conv dis}
			For constructible sheaves $F, G$ in $D^b(\Rbb^n)$, whose coefficients are not necessarily fields, $F\simeq G$ if $d_C(F, G)=0$. 
		\end{conjecture}
		Guillermou and Viterbo \cite{GV24} proved this conjecture for sheaves of $\Bk$-vector spaces. Since the Radon transform changes multi-directional movements to one-directional movements and keeps localized extended pseudo-distances, we expect that the property of the convolution distance $d_C$ on $D^b(\Rbb^n)$ is described by $d_I$ on $D^b(S^{n-1}\times\Rbb)$. In order to see behavior of $K_a*(-)$ and the Radon transform, let us consider an example below. 
		
		\begin{example}
		We recall the example in Section.\ref{rwmr}. 
		We use $S, I, J$ and $T$ to denote the subsets
			\begin{align*}
				S&:=\{(x_1, x_2)\in\Rbb^2; -1\leq x_1\leq 1, -1\leq x_2\leq 1\}, \\
				I&:=\{(x_1, 1)\in\Rbb^2; -1\leq x_1\leq 1\}, \\
				J&:=\{(x_1, -1)\in\Rbb^2; -1\leq x_1\leq 1\}, \\
				T&:=I\sqcup J
			\end{align*}
			of $\Rbb^2$ respectively. We set $Z=S\bs T$. By Proposition \ref{distinguish triangle}, we have
			\begin{equation}
			\label{dt_example}
				\Bk_{Z}\to \Bk_{S}\to \Bk_I\oplus\Bk_J\overset{+1}{\to}
			\end{equation}
			in $D^b(\Rbb^2)$. We note that $\Bk_T\simeq \Bk_I\oplus\Bk_J$. For each $a\geq0$, we set $S_a:=S+B_a=\{x+x^\prime \in\Rbb^2; x\in S, x^\prime\in B_a\}$. We remind that $B_a$ is a closed ball in $\Rbb^2$. 
			Similarly, we set $I_a:=I+B_a$ and $J_a:=J+B_a$. Let us define the subsets
			\begin{align*}
				T_a:= I_a\cup J_a, \quad
				Z^\prime_a:= S_a\bs T_a, \quad
				Z^{\prime\prime}_a:= I_a\cap J_a
			\end{align*}
			of $\Rbb^2$ respectively. Applying the functor $K_a*(-)$ to (\ref{dt_example}), we obtain the distinguish triangle in Figure 4. 

		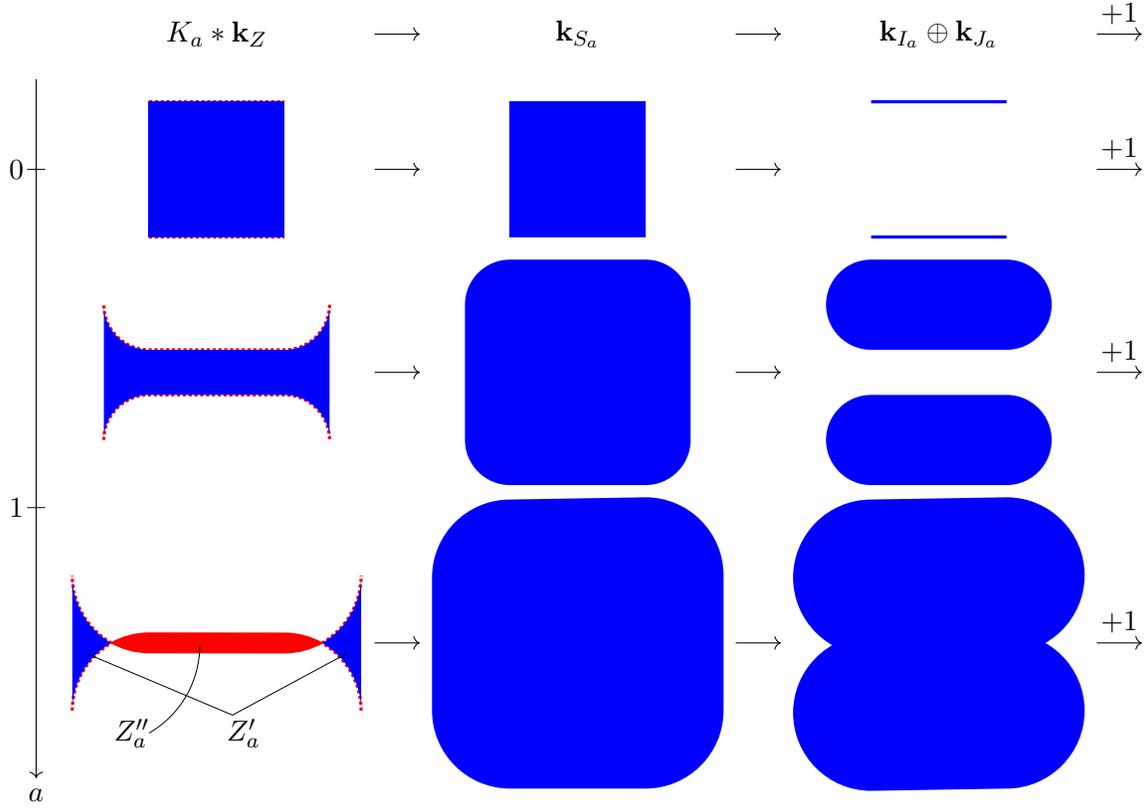
\begin{figure}[htbp]
		%\label{movement_example}
			\centering
			\begin{tikzpicture}[scale=0.3]
				\coordinate[label=left:$0$](0)at(-8.1, 0);
				\coordinate[label=left:$1$](1)at(-8.1, -15);
				\coordinate[label=below:$a$](a)at(-8, -27);
				\coordinate[label=center:$K_a*\Bk_Z$](KF)at(0, 6);
				\coordinate[label=center:$\Bk_{S_a}$](Sa)at(16, 6);
				\coordinate[label=center:$\Bk_{I_a}\oplus\Bk_{J_a}$](IaJa)at(32, 6);
				\coordinate[label=above:$+1$](+1)at(40, 6);
				\coordinate[label=above:$+1$](+10)at(40, 0);
				\coordinate[label=above:$+1$](+11)at(40, -9);
				\coordinate[label=above:$+1$](+12)at(40, -21);
				%\coordinate[label=left:$\GF_{\Bk_A}(K_a*\Bk_Z)$](RKF)at(-4, -8);
				\coordinate[label=right:$Z^{\prime\prime}_a$](Z’’a)at(-5,-25);
				\coordinate[label=right:$Z^\prime_a$](Z’a)at(0,-25);
				%矢印
				\draw[->] (7, 6) -- (9, 6);
				\draw[->] (23, 6) -- (25, 6);
				\draw[->] (39, 6) -- (41, 6);
				\draw[->] (7, 0) -- (9, 0);
				\draw[->] (23, 0) -- (25, 0);
				\draw[->] (39, 0) -- (41, 0);
				\draw[->] (7, -9) -- (9, -9);
				\draw[->] (23, -9) -- (25, -9);
				\draw[->] (39, -9) -- (41, -9);
				\draw[->] (7, -21) -- (9, -21);
				\draw[->] (23, -21) -- (25, -21);
				\draw[->] (39, -21) -- (41, -21);
				\draw[->] (-8, 4)--(-8, -27); 
				\draw[very thin] (-8.4, 0) -- (-7.6, 0);
				\draw[very thin] (-8.4, -15) -- (-7.6, -15);
				%Z
				\draw[red, very thick, densely  dotted] (-3, 3) -- (3, 3);
				\draw[red, very thick, densely  dotted] (-3, -3) -- (3, -3);
				\draw[red, very thick, densely  dotted] (-5, -12)arc(180:90:2)--(3, -10)arc(90:0:2);
				\draw[red, very thick, densely  dotted] (5, -6)arc(0:-90:2)--(-3, -8)arc(270:180:2);
				\draw[red, very thick, densely  dotted] (-6.4, -24)arc(180:120:3.46)arc(240:180:3.46);
				\draw[red, very thick, densely  dotted] (6.4, -24)arc(0:60:3.46)arc(-60:0:3.46);
				%S
				\fill[blue, opacity=0.4] (-3, 3)--(-3, -3)--(3, -3)--(3, 3)--cycle;
				\fill[blue, opacity=0.4] (-3, -8)arc(270:180:2)--(-5, -12)arc(180:90:2)--(3, -10)arc(90:0:2)--(5, -6)arc(0:-90:2)--cycle;
				\fill[blue, opacity=0.4] (-6.4, -24)arc(180:120:3.46)arc(240:180:3.46)--cycle;
				\fill[blue, opacity=0.4] (6.4, -24)arc(0:60:3.46)arc(-60:0:3.46)--cycle;
				\fill[red, opacity=0.4] (-3, -21.46)arc(270:240:3.46)arc(120:90:3.46)--(3, -20.54)arc(90:60:3.46)arc(-60:-90:3.46)--cycle;
				\fill[blue, opacity=0.4] (13, 3)--(13, -3)--(19, -3)--(19, 3)--cycle;
				\fill[blue, opacity=0.4] (13, -4)arc(90:180:2)--(11, -12)arc(180:270:2)--(19, -14)arc(270:360:2)--(21, -6)arc(0:90:2)--cycle;
				\fill[blue, opacity=0.4] (13, -14.64)arc(90:180:3.46)--(9.54, -24)arc(180:270:3.46)--(19, -27.46)arc(270:360:3.46)--(22.46, -18)arc(0:90:3.46)--cycle;
				%T
				\draw[blue, very thick] (29, 3) -- (35, 3);
				\draw[blue, very thick] (29, -3) -- (35, -3);
				\fill[blue, opacity=0.4] (29, -4)arc(90:270:2)--(35, -8)arc(270:450:2)--cycle;
				\fill[blue, opacity=0.4] (29, -10)arc(90:270:2)--(35, -14)arc(270:450:2)--cycle;
				\fill[blue, opacity=0.4] (29, -14.64)arc(90:270:3.46)--(35, -21.46)arc(270:450:3.46)--cycle;
				\fill[blue, opacity=0.4] (29, -20.64)arc(90:270:3.46)--(35, -27.46)arc(270:450:3.46)--cycle;
				\draw[very thin] (-3, -25) arc (-60:0:4.5); 
				\draw[very thin] (-5.7, -21.5)--(0.7, -24.2);
				\draw[very thin] (5.7, -21.5)--(0.7, -24.2);
			\end{tikzpicture}
			\caption{Applying $K_a*(-)$, we obtain the distinguish triangle. }
		\end{figure}
			Then we have
			\[
				H^0(K_a*\Bk_Z)\simeq \Bk_{Z^\prime_a},\quad 
				H^1(K_a*\Bk_Z)\simeq
				\begin{cases}
					0&\text{if }0\leq a<1, \\
					\Bk_{Z^{\prime\prime}_a}&\text{if }1\leq a.
				\end{cases}
			\]

			Applying the Radon transform, $\GF_{\Bk_A}(K_a*\Bk_Z)$ is simpler than $K_a*\Bk_Z$. For any $a\geq0$, there is a continuous function $\Gvf:S^1\to\Rbb$ such that $\GF_{\Bk_A}(K_a*\Bk_Z)\simeq\Bk_{\{(y, t)\in S^1\times\Rbb; \Gvf(y)\leq t+a\}}[-1]$(see Figure 5). 
		
		\begin{figure}[htbp]
		%\label{movement_example}
			\centering
			\begin{tikzpicture}[scale=0.3]
				\coordinate[label=below:$0$](0)at(0, -16.2);
				\coordinate[label=below:$1$](1)at(18, -16.2);
				\coordinate[label=right:$a$](a)at(31, -16);
				\coordinate[label=left:$K_a*\Bk_Z$](KF)at(-4, 0);
				\coordinate[label=left:$\GF_{\Bk_A}(K_a*\Bk_Z)$](RKF)at(-4, -8);
				\coordinate[label=right:$Z^{\prime\prime}_a$](Z’a)at(19,4);
				\coordinate[label=right:$Z^\prime_a$](Z’’a)at(24.9,4);
				\draw[very thin] (0, -15.6) -- (0, -16.4);
				\draw[very thin] (18, -15.6) -- (18, -16.4);
				\draw[->] (3.5, 0) -- (4.5, 0);
				\draw[->] (15.5, 0) -- (16.5, 0);
				\draw[red, very thick, densely  dotted] (-3, 3) -- (3, 3);
				\draw[red, very thick, densely  dotted] (-3, -3) -- (3, -3);
				\draw[red, very thick, densely  dotted] (5, -3)arc(180:90:2)--(13, -1)arc(90:0:2);
				\draw[red, very thick, densely  dotted] (15, 3)arc(0:-90:2)--(7, 1)arc(270:180:2);
				\draw[red, very thick, densely  dotted] (17.6, -3)arc(180:120:3.46)arc(240:180:3.46);
				\draw[red, very thick, densely  dotted] (30.4, -3)arc(0:60:3.46)arc(-60:0:3.46);
				\draw[->] (4, -9) -- (6, -9);
				\draw[->] (16, -9) -- (18, -9);
				\draw[->] (-4, -16) -- (31, -16);
				\draw[very thin] (21, 4) arc (60:0:4.5); 
				\draw[very thin] (18, 0)--(25, 4);
				\draw[very thin] (30, 0)--(27.2, 4);
				\fill[blue, opacity=0.4] (-3, 3)--(-3, -3)--(3, -3)--(3, 3)--cycle;
				\fill[blue, opacity=0.4] (7, 1)arc(270:180:2)--(5, -3)arc(180:90:2)--(13, -1)arc(90:0:2)--(15, 3)arc(0:-90:2)--cycle;
				\fill[blue, opacity=0.4] (17.6, -3)arc(180:120:3.46)arc(240:180:3.46)--cycle;
				\fill[blue, opacity=0.4] (30.4, -3)arc(0:60:3.46)arc(-60:0:3.46)--cycle;
				\fill[red, opacity=0.4] (21, -0.46)arc(270:240:3.46)arc(120:90:3.46)--(27, 0.46)arc(90:60:3.46)arc(-60:-90:3.46)--cycle;
				%\fill[blue, opacity=0.4] (-3, 2)arc(135:225:2.8284)--(-2, -3)arc(225:315:2.8284)--(3, -2)arc(-45:45:2.8284)--(2, 3)arc(45:135:2.8284)--cycle;
				%\fill[gray, opacity=0.4] (14, -7) circle (3mm);
				%\fill[blue, opacity=0.4] (8, -7) circle (14.142mm);
				%\fill[blue, opacity=0.4] (0, -7) circle (28.284mm);
				\foreach\a in {-3, 7, 21}{
				\coordinate[label=below:$\Rbb$](Rbb1)at(\a, -4);
				\coordinate[label=below:$\Rbb$](Rbb1)at(\a+6, -4);
				\coordinate[label=right:$S^1$](S1)at(\a+6.5, -10);
				\draw[very thin] (\a, -14) -- (\a, -6);
				\draw[very thin] (\a +6, -14) -- (\a +6, -6);
				\draw[thin] (\a, -10) -- (\a+6, -10);
				\fill[red, opacity=0.4]
					plot[domain=0:90,smooth]({\a+(\x)/60},{2.82*sin((\x)-45)-10-(\a>6)-(\a>20)})
					--plot[domain=90:180,smooth]({\a+(\x)/60},{2.82*sin((\x)+45)-10-(\a>6)-(\a>20)}) 
					--plot[domain=180:270,smooth]({\a+(\x)/60},{2.82*sin((\x)+135)-10-(\a>6)-(\a>20)}) 
					--plot[domain=270:360,smooth]({\a+(\x)/60},{2.82*sin((\x)+225)-10-(\a>6)-(\a>20)})
					--(\a+6, -6.4)--(\a, -6.4)--cycle; 
				}
			\end{tikzpicture}

				\caption{This figure is the same as Figure 1 in Section.\ref{rwmr} . For each $a\geq0$, the upper shape is $Z^\prime_a$ or $Z^\prime_a \cup Z^{\prime\prime}_a$ and the lower shape is $\{(y, t)\in S^1\times\Rbb; \Gvf(y)\leq t+a\}$. Since $S^1$ is a quotient space of the closed interval $[0, 1]$ of $\Rbb$ by identifying its endpoints $\{0, 1\}$, we show $\{(y, t)\in S^1\times\Rbb; \Gvf(y)\leq t+a\}$ as a subset of $[0, 1]\times\Rbb$ on this figure. }
			\end{figure}
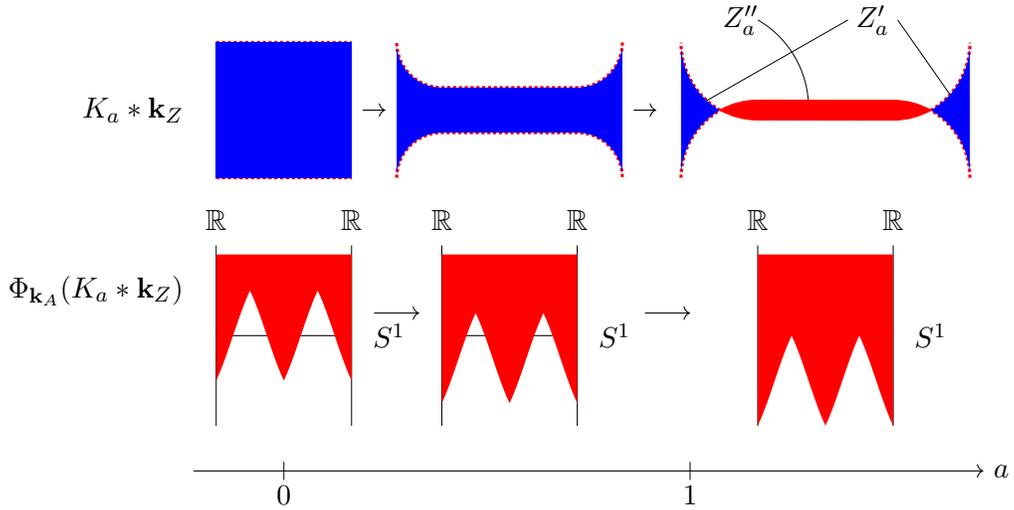

			Increasing $a\geq0$, $Z^\prime_a$ is separated into two connected components and $H^1(K_a*\Bk_{Z})$ appears. On the other hand, increasing $a\geq0$, $\{(y, t)\in S^1\times\Rbb; \Gvf(y)\leq t+a\}$ just moves in the direction of $\Rbb$. Then by Lemma \ref{commutative between radon and cmp}, the movements given by the functor $K_a*(-)$ can be replaced with the tractable movements by the functor $L_a\tcv(-)$. 
		\end{example}

%		When $n=1$, this problem is already solved by \cite{BG22}. %If the following lemma is proved, the closedness of $d_C$ on constructible sheaves on $\Rbb^n$. 
%		\begin{proposition}\cite[Exercise VIII.7.]{KS90}
%		\label{qct-constructible}
%			Suppose $K\in D^b_{\Rbb_c}(X\times Y)$ is a quantized contact transform from $(X, \GO_X)$ to $(Y, \GO_Y)$. Then
%			\begin{equation}
%			 	\GF_K:D^b_{\Rbb_c}(X; \GO_X)\to D^b_{\Rbb_c}(Y; \GO_Y)
%			\end{equation}
%			is an equivalence of categories. 
%		\end{proposition}

		Our main theorem \ref{re:main theorem} gives the following conjecture from Conjecture \ref{conj conv dis}. 
		\begin{conjecture}
		\label{conj int dis}
			For constructible sheaves $F, G$ in $D^b(S^{n-1}\times\Rbb)$, whose coefficients are not necessarily fields, $F\simeq G$ if $d_I(F, G)=0$. 
		\end{conjecture}
		Since $L_a\tcv(-)$ gives simpler movements than $K_a*(-)$, it may be easier to solve this conjecture than Conjecture \ref{conj conv dis}.

\end{document}